\documentclass[english,twoside]{article}
\usepackage[T1]{fontenc}
\usepackage[utf8]{inputenc}
\usepackage[english]{babel}
\usepackage[tmargin=2.5cm,bmargin=2.5cm,lmargin=3cm,rmargin=2.5cm]{geometry}
\usepackage{
  amsthm,amsmath,amssymb,esint,graphicx,color,
  hyperref,placeins,caption,cleveref,authblk
}

\newtheorem{theorem}{Theorem}

\newtheorem{lemma}[theorem]{Lemma}

\newcommand{\mat}[1]{\boldsymbol{#1}}
\newcommand{\dd}{\; d}
\newcommand{\abs}[1]{\left|#1\right|}
\newcommand{\norm}[1]{\left|\!\left|#1\right|\!\right|}
\newcommand{\restr}[2]{{
    \left.\kern-\nulldelimiterspace 
      #1 
      \vphantom{\big|} 
    \right|_{#2} 
  }}
\newcommand{\grad}{\nabla}
\renewcommand{\div}{\nabla\cdot}
\newcommand{\ip}[2]{\left(#1,#2\right)}
\newcommand{\pair}[2]{\left\langle #1,#2\right\rangle}
\renewcommand{\hat}[1]{\widehat{#1}}
\renewcommand{\tilde}[1]{\widetilde{#1}}

\usepackage[
backend=biber,
style=numeric,
natbib=true,
url=false,
doi=true,
eprint=false,
sortcites=true,
maxbibnames=99
]{biblatex}
\AtEveryBibitem{%
  \clearfield{note}%
}
\author[1,2]{Mark Ainsworth}
\author[1,3]{Christian Glusa}
\affil[1]{Division of Applied Mathematics, Brown University, 182 George St, Providence, RI 02912, USA}
\affil[2]{Computer Science and Mathematics Division, Oak Ridge National Laboratory, Oak Ridge, TN 37831, USA}
\affil[3]{Center for Computing Research, Sandia National Laboratories, Albuquerque, NM 87185, USA\thanks{Sandia National Laboratories is a multimission laboratory managed and operated by National Technology and Engineering Solutions of Sandia, LLC, a wholly owned subsidiary of Honeywell International, Inc., for the U.S. Department of Energy’s National Nuclear Security Administration under contract DE-NA0003525. \newline
    SAND Number: SAND2017-9468 O}}

\title{Hybrid Finite Element - Spectral Method for the Fractional Laplacian:
  Approximation Theory and Efficient Solver
  \thanks{This work was supported by the MURI/ARO on ``Fractional PDEs for Conservation Laws and Beyond: Theory, Numerics and Applications'' (W911NF-15-1-0562).}
  }

\begin{document}
\maketitle

\abstract{
  A numerical scheme is presented for approximating fractional order Poisson problems in two and three dimensions.
  The scheme is based on reformulating the original problem posed over \(\Omega\) on the extruded domain \(\mathcal{C}=\Omega\times[0,\infty)\) following \cite{CaffarelliSilvestre2007_ExtensionProblemRelatedToFractionalLaplacian}.
  The resulting degenerate elliptic \emph{integer} order PDE is then approximated using a hybrid FEM-spectral scheme.
  Finite elements are used in the direction parallel to the problem domain \(\Omega\), and an appropriate spectral method is used in the extruded direction.
  The spectral part of the scheme requires that we approximate the true eigenvalues of the integer order Laplacian over \(\Omega\).
  We derive an a priori error estimate which takes account of the error arising from using an approximation in place  of the true eigenvalues.
  We further present a strategy for choosing approximations of the eigenvalues based on Weyl's law and finite element discretizations of the eigenvalue problem.
  The system of linear algebraic equations arising from the hybrid FEM-spectral scheme is decomposed into blocks which can be solved effectively using standard iterative solvers such as multigrid and conjugate gradient.
  Numerical examples in two and three dimensions show that the approach is quasi-optimal in terms of complexity.
}

\section{Introduction}
\label{sec:introduction}

Over the last few years, \emph{non-local} and \emph{fractional order} models models have seen a surge in interest in a wide variety of application areas such as anomalous diffusion, material science, image processing, finance and electromagnetic fluids \cite{West2016_FractionalCalculusViewComplexity}.
Compared with local, integer order equations, the linear algebraic systems arising from fractional order models are generally dense, and can be difficult to solve efficiently.
In the present work, we explore how structural sparsity can be leveraged to solve a fractional order Poisson problem in quasi-optimal complexity.

Let \(\Omega\in C^{2}\) or a convex polyhedron in \(\mathbb{R}^{d}\).
One of the many possible ways of defining a fractional order Laplacian on \(\Omega\) uses the spectral information of the integer order operator.
Let \(0<\lambda_{0}\leq\lambda_{1}\leq\dots\) and \(\phi_{0},\phi_{1},\dots\) be the eigenvalues and eigenfunctions of the regular Laplacian
\begin{align}
  \left\{
  \begin{array}{rlrl}
    -\Delta\phi_{m}\left(\vec{x}\right)&=\lambda_{m}\phi_{m}\left(\vec{x}\right), &&\vec{x}\in\Omega, \\
    \phi_{m}\left(\vec{x}\right) &=0, && \vec{x}\in\partial\Omega,
  \end{array}\right. \tag{Eig}\label{eq:Eig}
\end{align}
normalised so that \(\norm{\phi_{m}}_{L^{2}}=1\).
The eigenfunctions \(\left\{\phi_{m}\right\}_{m=0}^{\infty}\) form a complete orthonormal basis of \(L^{2}\left(\Omega\right)\).
This means that any function \(u\in L^{2}\left(\Omega\right)\) can be expanded as
\begin{align}
  u=\sum_{m=0}^{\infty}u_{m}\phi_{m} \qquad\text{with } u_{m}= \ip{u}{\phi_{m}}_{L^{2}}. \label{eq:orthoExpansion}
\end{align}
In particular,
\begin{align*}
  \left(-\Delta\right) u\left(\vec{x}\right)&= \sum_{m=0}^{\infty}u_{m}\lambda_{m}\phi_{m}\left(\vec{x}\right),
\end{align*}
while the \emph{spectral fractional Laplacian} of order \(s\in(0,1)\) is given by
\begin{align*}
  \left(-\Delta\right)^{s}u\left(\vec{x}\right)
  &= \sum_{m=0}^{\infty} u_{m}\lambda_{m}^{s} \phi_{m}\left(\vec{x}\right).
\end{align*}
As \(s\rightarrow 0\), the identity is recovered, whereas the usual, integer order Laplacian is recovered as \(s\rightarrow1\).

We are interested in solving the fractional order Poisson problem
\begin{align}
  \left\{
  \begin{array}{rlrl}
    \left(-\Delta\right)^{s}u\left(\vec{x}\right) &= f\left(\vec{x}\right), && \vec{x}\in\Omega,\\
    u(\vec{x})&=0, && \vec{x}\in\partial\Omega
  \end{array}\right. \tag{fP}\label{eq:fracPoisson}
\end{align}
with given right-hand side \(f\).

The spectral definition is not the only possibility to define a fractional order Laplacian on \(\Omega\); other choices include the so-called \emph{integral fractional Laplacian}, defined as
\begin{align*}
  \left(-\Delta\right)_{I}^{s} u\left(\vec{x}\right) = C(d,s) \operatorname{p.v.} \int_{\mathbb{R}^{d}} \dd \vec{y} ~ \frac{u(\vec{x})-u(\vec{y})}{\abs{\vec{x}-\vec{y}}^{d+2s}}
\end{align*}
where
\begin{align*}
  C(d,s) = \frac{2^{2s}s\Gamma\left(s+\frac{d}{2}\right)}{\pi^{d/2}\Gamma\left(1-s\right)}
\end{align*}
is a normalisation constant and \(\operatorname{p.v.}\) denotes the Cauchy principal value of the integral \cite[Chapter 5]{Mclean2000_StronglyEllipticSystemsBoundaryIntegralEquations}.
If \(\Omega=\mathbb{R}^{d}\), the two definitions coincide, but they are different for bounded domains \cite{ServadeiValdinoci2014_SpectrumTwoDifferentFractionalOperators}.
In previous work \cite{AinsworthGlusa2017_AspectsAdaptiveFiniteElement,AinsworthGlusa2017_TowardsEfficientFiniteElement}, we demonstrated that adaptive finite elements and multigrid methods can be used to solve fractional equations based on the integral definition in quasi-optimal complexity.

Existing solution methods for the fractional Poisson problem involving the spectral definition of the fractional Laplacian generally follow one of two different paths:
exploit the Dunford-Taylor integral representation of the fractional power of the discretized integral order Laplacian \cite{BonitoPasciak2015_NumericalApproximationFractionalPowersEllipticOperators,BonitoPasciak2016_NumericalApproximationFractionalPowers};
or, interpret the fractional operator as a Dirichlet-to-Neumann map of a singular elliptic problem embedded in \(d+1\) space dimensions - the so-called \emph{extruded problem} approach \cite{CaffarelliSilvestre2007_ExtensionProblemRelatedToFractionalLaplacian,StingaTorrea2010_ExtensionProblemHarnacksInequalityFractionalOperators}.

\citeauthor{NochettoOtarolaEtAl2015_PdeApproachToFractional} \cite{NochettoOtarolaEtAl2015_PdeApproachToFractional} used finite elements to discretize the extruded problem, along with a problem specific multigrid solver \cite{ChenNochettoEtAl2016_MultilevelMethodsNonuniformlyElliptic}, while \cite{MeidnerPfeffererEtAl2017_HpFiniteElementsFractionalDiffusion} is based on a \(hp\)-FEM discretization.

In this work, we also approximate the extruded problem.
However, while we use finite elements in the direction parallel to the problem domain, we introduce a spectral method in the extruded direction.
By a careful choice of expansion functions in the spectral method, we recover a quasi-optimal method.

This work is structured as follows:
In \Cref{sec:notation}, we introduce the necessary notation as well as the extruded problem associated with the spectral fractional Laplacian.
We briefly discuss the eigenfunctions of the extruded problem in \Cref{sec:eigenf-extend-probl}, which are then used to discretize the extruded problem in \Cref{sec:probl-discr}.
In \Cref{sec:error-bound}, we derive an a priori error bound which is explicit in the mesh size \(h\) on \(\Omega\) and the spectral order on \([0,\infty)\).
The method requires suitable approximation of the true eigenvalues of the standard Laplacian over \(\Omega\).
We describe in \Cref{sec:choice-appr-eigenv} how the approximations can be obtained in an efficient manner.
In \Cref{sec:solution-linear-system}, we give details on the solution of the resulting linear systems using a multigrid solver.
Finally, in \Cref{sec:numerical-examples}, numerical results are presented that confirm quasi-optimal complexity of our algorithm.

\section{Notation}
\label{sec:notation}

Let \(\Omega\) be a subdomain of \(\mathbb{R}^{d}\) as above, then \cite{Mclean2000_StronglyEllipticSystemsBoundaryIntegralEquations} we define the Sobolev space \(H^{s}\left(\Omega\right)\) to be
\begin{align*}
   H^{s}\left(\Omega\right)&:=\left\{u\in L^{2}\left(\Omega\right) \mid \norm{u}_{H^{s}\left(\Omega\right)} < \infty\right\},
\end{align*}
equipped with the norm
\begin{align*}
  \norm{u}_{H^{s}\left(\Omega\right)}^{2}&= \norm{u}_{L^{2}\left(\Omega\right)}^{2} + \int_{\Omega}\dd \vec{x} \int_{\Omega}\dd \vec{y} \frac{\left(u(\vec{x})-u(\vec{y})\right)^{2}}{\abs{\vec{x}-\vec{y}}^{d+2s}}.
\end{align*}
The space \(\tilde{H}^{s}\left(\Omega\right)\) is defined as \cite[Appendix B]{Bramble1993_MultigridMethods}
\begin{align*}
  \tilde{H}^{s}\left(\Omega\right)&=\left\{u\in L^{2}\left(\Omega\right) \mid \abs{u}_{\tilde{H}^{s}}<\infty\right\},
\end{align*}
where the norm is given by
\begin{align*}
  \abs{u}_{\tilde{H}^{s}}^{2}&= \sum_{m=0}^{\infty}u_{m}^{2}\lambda_{m}^{s},
\end{align*}
where \(u_{m}\) are defined in \eqref{eq:orthoExpansion}.
For \(s>1/2\), \(\tilde{H}^{s}\left(\Omega\right)\) coincides with the space \(H_{0}^{s}\left(\Omega\right)\) defined to be the closure of \(C_{0}^{\infty}\left(\Omega\right)\) with respect to the \(H^{s}\left(\Omega\right)\)-norm, whilst for \(s<1/2\), \(\tilde{H}^{s}\left(\Omega\right)\) is identical to \(H^{s}\left(\Omega\right)\).
In the critical case \(s=1/2\), \(\tilde{H}^{s}\left(\Omega\right)\subset H^{s}_{0}\left(\Omega\right)\), and the inclusion is strict.
(See for example \cite[Chapter 3]{Mclean2000_StronglyEllipticSystemsBoundaryIntegralEquations}.)

The spaces \(\tilde{H}^{s}\left(\Omega\right)\) are a useful vehicle to describe the properties of the spectral fractional Laplacians:
For instance, suppose \(f\in\tilde{H}^{r}\left(\Omega\right)\), \(r\geq -s\), and \(f=\sum_{m=0}^{\infty}f_{m}\phi_{m}\left(\vec{x}\right)\) with \(f_{k}=\ip{f}{\phi_{m}}_{L^{2}}\) then the solution \(u\) to the fractional Poisson problem \(\eqref{eq:fracPoisson}\) of order \(s\) with right-hand side \(f\) is given by
\begin{align}
  u&=\sum_{m=0}^{\infty} u_{m}\phi_{m}(\vec{x}), \qquad u_{m}=f_{m}\lambda_{m}^{-s}, \label{eq:orthoExpansionFracLapl}
\end{align}
and hence \(u\in \tilde{H}^{r+2s}\left(\Omega\right)\).
A more detailed regularity theory for spectral Poisson problems can be found in the work of \citeauthor{Grubb2015_RegularitySpectralFractionalDirichletNeumannProblems} \cite{Grubb2015_RegularitySpectralFractionalDirichletNeumannProblems}.

We also define the weighted norms on a generic domain \(\mathcal{D}\) for a non-negative weight function \(\omega\) by
\begin{align*}
  \norm{u}_{L^{2}_{\omega}}^{2} &= \int_{\mathcal{D}} \omega \abs{u}^{2}, &
  \abs{u}_{H^{1}_{\omega}}^{2} &= \int_{\mathcal{D}} \omega \abs{\grad u}^{2}, \\
  \norm{u}_{H^{1}_{\omega}}^{2} &= \norm{u}_{L^{2}_{\omega}}^{2}+\abs{u}_{H^{1}_{\omega}}^{2},
\end{align*}
along with the associated weighted spaces
\begin{align*}
  L^{2}_{\omega}\left(\mathcal{D}\right) &= \left\{u \text{ measurable } \mid \norm{u}_{L^{2}_{\omega}}<\infty\right\}, &
  H^{1}_{\omega}\left(\mathcal{D}\right)&= \left\{u\in L^{2}_{\omega}\left(\mathcal{D}\right) \mid \norm{u}_{H^{1}_{\omega}}<\infty\right\}.
\end{align*}

Building on the work of \citeauthor{CaffarelliSilvestre2007_ExtensionProblemRelatedToFractionalLaplacian} \cite{CaffarelliSilvestre2007_ExtensionProblemRelatedToFractionalLaplacian}, \citeauthor{StingaTorrea2010_ExtensionProblemHarnacksInequalityFractionalOperators} \cite{StingaTorrea2010_ExtensionProblemHarnacksInequalityFractionalOperators} showed that the fractional Poisson problem \(\eqref{eq:fracPoisson}\) can be recast as a problem over the extruded domain \(\mathcal{C}=\Omega\times[0,\infty)\):
\begin{align}
  \left\{
  \begin{array}{rlrl}
    -\div y^{\alpha} \grad U\left(\vec{x},y\right) &= 0, && \left(\vec{x},y\right)\in\mathcal{C}, \\
    U\left(\vec{x},y\right) &= 0, && \left(\vec{x},y\right)\in\partial_{L}\mathcal{C} := \partial\Omega\times[0,\infty), \\
    \frac{\partial U}{\partial \nu^{\alpha}}\left(\vec{x}\right) &= d_{s}f\left(\vec{x}\right), && \vec{x}\in \Omega,
  \end{array} \right. \tag{Ext}\label{eq:extendedProblem}
\end{align}
where \(\alpha = 1-2s\), \(d_{s} = 2^{1-2s}\frac{\Gamma\left(1-s\right)}{\Gamma\left(s\right)}\), and
\begin{align*}
  \frac{\partial U}{\partial \nu^{\alpha}}\left(\vec{x}\right)&= -\lim_{y\rightarrow 0^{+}}y^{\alpha}\frac{\partial U}{\partial y}\left(\vec{x},y\right),
\end{align*}
with the solution to \(\eqref{eq:fracPoisson}\) recovered by taking the trace of \(U\) on \(\Omega\), i.e. \(u=\operatorname{tr}_{\Omega}U\).

We define the solution space \(\mathcal{H}^{1}_{\alpha}\left(\mathcal{C}\right)\) on the semi-infinite cylinder \(\mathcal{C}\) as
\begin{align*}
  \mathcal{H}_{\alpha}^{1}\left(\mathcal{C}\right)&=\left\{V\in H^{1}_{y^{\alpha}}\left(\mathcal{C}\right) \mid V=0 \text{ on } \partial_{L}\mathcal{C}\right\},
\end{align*}
with norm \(\norm{V}_{\mathcal{H}^{1}_{\alpha}}= \abs{V}_{H^{1}_{y^{\alpha}}}\).
The weak formulation of the extruded problem \(\eqref{eq:extendedProblem}\) consists of seeking \(U\in\mathcal{H}_{\alpha}^{1}\left(\mathcal{C}\right)\) such that:
\begin{align}
  \int_{{\cal C}}y^{\alpha}\grad U\cdot \grad V &= d_{s}\pair{f}{\operatorname{tr}_{\Omega}V} \quad \forall V \in\mathcal{H}_{\alpha}^{1}\left(\mathcal{C}\right). \label{eq:var}
\end{align}
Using a trace inequality \cite{NochettoOtarolaEtAl2015_PdeApproachToFractional}, the Lax-Milgram Lemma shows that the extruded problem is well-posed.

\section{Eigenfunctions of the Extruded Problem}
\label{sec:eigenf-extend-probl}

We seek a solution of the extruded problem using classical separation of variables: \(U\left(\vec{x},y\right)=\Phi\left(\vec{x}\right) \Psi\left(y\right)\).
Then
\begin{align*}
  \frac{-\Delta_{\vec{x}}\Phi}{\Phi}=\frac{\partial_{y}y^{\alpha}\partial_{y}\Psi}{y^{\alpha}\Psi}=A,
\end{align*}
where \(A\) is a constant that is independent of \(\vec{x}\) and \(y\).
The boundary condition on the lateral face of the cylinder \(\mathcal{C}\), shows that \(\Phi=\phi_{m}\) and \(A=\lambda_{m}\) for \(m\in\mathbb{N}\) thanks to \(\eqref{eq:Eig}\).
The associated function \(\Psi\) in the extruded direction must therefore satisfy
\begin{align}
  \partial_{y}y^{\alpha}\partial_{y}\Psi=\lambda_{m}y^{\alpha}\Psi, \label{eq:extendedODE}
\end{align}
or, equivalently,
\begin{align*}
  \partial_{y}^{2}\Psi + \frac{\alpha}{y}\partial_{y}\Psi - \lambda_{m}\Psi=0.
\end{align*}
Choosing the normalisation \(\Psi\left(0\right)=1\) gives
\begin{align}
  \Psi\left(y\right)=\psi_{m}\left(y\right)
  &:= c_{s}\left(\lambda_{m}^{1/2}y\right)^{s} K_{s}\left(\lambda_{m}^{1/2}y\right),   \label{eq:spectralBasis}
\end{align}
where \(c_{s}= 2^{1-s} / \Gamma\left(s\right)\).
Moreover
\begin{align*}
  \frac{\partial \psi_{m}}{\partial \nu^{\alpha}}= d_{s}\lambda_{m}^{s},
\end{align*}
so that
\begin{align}
  \int_{0}^{\infty}y^{\alpha} \psi_{m}\psi_{n}
  &=
    \begin{cases}
      d_{s}\frac{\lambda_{m}^{s}-\lambda_{n}^{s}}{\lambda_{m}-\lambda_{n}} & \text{if } m\neq n,\\
      sd_{s}\lambda_{m}^{s-1} & \text{if } m=n,
    \end{cases} \label{eq:massSpec}
                                \intertext{and}
  \int_{0}^{\infty}y^{\alpha} \psi_{m}'\psi_{n}'
  &=
    \begin{cases}
      d_{s}\frac{\lambda_{m}\lambda_{n}^{s}-\lambda_{n}\lambda_{m}^{s}}{\lambda_{m}-\lambda_{n}} & \text{if } m\neq n,\\
      (1-s)d_{s}\lambda_{m}^{s} & \text{if } m=n.
    \end{cases}\label{eq:stiffnesSpec}
\end{align}
The solution to the extruded problem \(\eqref{eq:extendedProblem}\) is then given by
\begin{align}
  U\left(\vec{x},y\right)&=\sum_{m=0}^{\infty} u_{m}\phi_{m}(\vec{x})\psi_{m}\left(y\right) \quad \text{where } u_{m}=\lambda_{m}^{-s}f_{m}, \label{eq:separableSol}
\end{align}
whilst \(u\left(\vec{x}\right)=\sum_{m=0}^{\infty}u_{m}\phi_{m}\left(\vec{x}\right)\) as in \eqref{eq:orthoExpansionFracLapl}.
The separable solution \eqref{eq:separableSol} forms the basis for our choice of discretization of the extruded problem to be described in the next section.
The chief advantage of this approach is that the extruded problem involves only integer order derivatives but come as the price of having to deal with a degenerate weight \(y^{\alpha}\).

\section{Discretization of the Extruded Problem}
\label{sec:probl-discr}

We propose to approximate the variational problem \eqref{eq:var} using a Galerkin scheme with the subspace consisting of standard low order nodal finite elements of order \(k\) in the \(\vec{x}\)-variable and a spectral method in the \(y\)-direction.
To this end, we let \(\mathcal{T}_{h}\) be a shape regular, globally quasi-uniform triangulation of \(\Omega\), and let
\begin{align*}
  V_{h}=\left\{u_{h}\in H^{1}_{0}\left(\Omega\right) \mid \restr{u_{h}}{K}\in \mathbb{P}_{k}\left(K\right) ~\forall K\in\mathcal{T}_{h}\right\}.
\end{align*}
Ideally, we would like to use \(y\)-basis functions given by \eqref{eq:spectralBasis}.
Unfortunately, this would require knowledge of the true eigenvalues of the integer order Laplacian over \(\Omega\).
Instead, for a given spectral expansion order \(M\), we use an approximation \(\tilde{\lambda}_{m} \approx \lambda_{m}\) in place of the true eigenvalues in \eqref{eq:spectralBasis}:
\begin{align}
  \tilde{\psi}_{m}\left(y\right)
  &:= c_{s}\left(\tilde{\lambda}_{m}^{1/2}y\right)^{s} K_{s}\left(\tilde{\lambda}_{m}^{1/2}y\right).
\end{align}
The Galerkin subspace for the extruded problem is then taken to be
\begin{align*}
  \mathcal{V}_{h,M}=\left\{U_{h,M}=\sum_{m=0}^{\tilde{M}-1}u_{h,m}\left(\vec{x}\right)\tilde{\psi}_{m}\left(y\right) \mid u_{h,m}\in V_{h} \right\}\subset \mathcal{H}_{\alpha}^{1}\left(\mathcal{C}\right).
\end{align*}
The selection of the approximate eigenvalues is discussed in \Cref{sec:choice-appr-eigenv}.
In particular, if two or more of the approximate eigenvalues are ``close'' then we retain only a single eigenvalue, thereby reducing the dimension of \(\mathcal{V}_{h,M}\) to \(\mathcal{N}:=\dim \mathcal{V}_{h,M}=n\times \tilde{M}\), where \(n:=\dim V_{h}\) and \(\tilde{M}\leq M\) is the number of distinct approximate eigenvalues.
We return to this point in \Cref{sec:size-reduct-appr}.
In the analysis, it will be useful to consider the semi-discrete space
\begin{align*}
  \mathcal{V}_{M}=\left\{U_{M}=\sum_{m=0}^{M-1}u_{m}\left(\vec{x}\right)\tilde{\psi}_{m}\left(y\right) \mid u_{m}\in H^{1}_{0}\left(\Omega\right) \right\} \subset \mathcal{H}_{\alpha}^{1}\left(\mathcal{C}\right).
\end{align*}

The Galerkin approximation consists of seeking \(U_{h,M} \in\mathcal{V}_{h,M}\) such that
\begin{align}
  \int_{{\cal C}}y^{\alpha}\grad U_{h,M}\cdot \grad V &= d_{s}\pair{f}{\operatorname{tr}_{\Omega}V} \quad \forall V \in\mathcal{V}_{h,M}, \label{eq:varSubspace}
\end{align}
with the approximation of the fractional Poisson problem given by
\begin{align*}
  u_{h,M}&:=\operatorname{tr}_{\Omega}U_{h,M}.
\end{align*}
We wish to obtain an estimate for the error \(u-u_{h,M}\) in this approximation.
The trace inequality \cite{NochettoOtarolaEtAl2015_PdeApproachToFractional} implies that
\begin{align*}
  \norm{u-u_{h,M}}_{\tilde{H}^{s}} &\leq C \norm{U-U_{h,M}}_{\mathcal{H}^{1}_{\alpha}},
\end{align*}
where the constant is independent of \(k\), \(M\) and \(h\).
Hence, in order to bound \(u-u_{h,M}\), it suffices to bound the term on the right-hand side - the discretization error of the extruded problem \eqref{eq:varSubspace}.

\section{A Priori Error Estimate}
\label{sec:error-bound}

We first consider the error in the approximation given by the semi-discrete Galerkin scheme on the space \(\mathcal{V}_{M}\).
The following result shows how the error depends on \(M\) and on how well the approximate eigenvalues \(\left\{\tilde{\lambda}_{m}\right\}\) match the true eigenvalues \(\left\{\lambda_{m}\right\}\).

\begin{lemma}\label{lem:semianalyticBound}
  Let \(M\in\mathbb{N}\) and \(U\in\mathcal{H}^{1}_{\alpha}\left(\mathcal{C}\right)\) be the solution of the extruded problem.
  Then
  \begin{align*}
    \inf_{V_{M}\in\mathcal{V}_{M}}\norm{U-V_{M}}_{\mathcal{H}^{1}_{\alpha}}^{2}&= d_{s}\sum_{m=0}^{\infty}\beta_{m}u_{m}^{2}\lambda_{m}^{s},
                                                   \intertext{where}
    \beta_{m}&=
    \begin{cases}
      g\left(s, \tilde{\lambda}_{m} / \lambda_{m}\right) & m=0,\dots,M-1, \\
      1 & m\geq M,
    \end{cases}
          \intertext{and}
          g\left(s,  \rho\right)&=1-\frac{1}{(1-s)\rho^{s} + s \rho^{s-1}}.
  \end{align*}
\end{lemma}
\begin{proof}
  Without loss of generality, we may write \(U=\sum_{m=0}^{\infty} u_{m}\phi_{m}(\vec{x})\psi_{m}\left(y\right)\), and consider \(V_{M}=\sum_{m=0}^{M-1} \alpha_{m}u_{m}\phi_{m}(\vec{x})\tilde{\psi}_{m}\left(y\right)\), where \(\alpha_{m}\in\mathbb{R}\) will be chosen below.
  Direct computation gives
  \begin{align*}
    \norm{U-V_{M}}_{\mathcal{H}_{\alpha}^{1}}^{2}
    &= \sum_{m=0}^{M-1}\sum_{n=0}^{M-1} u_{m}u_{n} \pair{\phi_{m}\left(\psi_{m}-\alpha_{m}\tilde{\psi}_{m}\right)}{\phi_{n}\left(\psi_{n}-\alpha_{n}\tilde{\psi}_{n}\right)}_{\mathcal{H}_{\alpha}^{1}} \\
      &\quad + 2\sum_{m=0}^{M-1}\sum_{n=M}^{\infty} u_{m}u_{n} \pair{\phi_{m}\left(\psi_{m}-\alpha_{m}\tilde{\psi}_{m}\right)}{\phi_{n}\psi_{n}}_{\mathcal{H}_{\alpha}^{1}} \\
      & \quad+ \sum_{m=M}^{\infty}\sum_{n=M}^{\infty} u_{m}u_{n} \pair{\phi_{m}\psi_{m}}{\phi_{n}\psi_{n}}_{\mathcal{H}_{\alpha}^{1}}.
  \end{align*}
  To deal with the first term, we observe that for arbitrary smooth functions \(h_{1}\) and \(h_{2}\) there holds
  \begin{align*}
    \pair{\phi_{m}\left(\vec{x}\right) h_{1}\left(y\right)}{\phi_{n}\left(\vec{x}\right)h_{2}\left(y\right)}_{\mathcal{H}_{\alpha}^{1}}
    &= \int_{\mathcal{C}}y^{\alpha} \grad \left[\phi_{m}\left(\vec{x}\right) h_{1}\left(y\right)\right]\cdot \grad \left[\phi_{n}\left(\vec{x}\right) h_{2}\left(y\right)\right] \\
    &= \int_{\Omega} \phi_{m} \phi_{n} \int_{0}^{\infty} y^{\alpha} h_{1}'' h_{2}' +
      \int_{\Omega} \grad_{\vec{x}}\phi_{m}\cdot \grad_{\vec{x}} \phi_{n} \int_{0}^{\infty} y^{\alpha} h_{1} h_{2} \\
    &= \delta_{nm}\ip{h_{1}}{h_{2}}_{m}
  \end{align*}
  where the inner product in the final equality is defined to be
  \begin{align*}
    \ip{h_{1}}{h_{2}}_{m} &= \int_{0}^{\infty}y^{\alpha} h_{1}' h_{2}' + \lambda_{m} \int_{0}^{\infty}y^{\alpha} h_{1}h_{2},
  \end{align*}
  with the induced norm denoted by \(\norm{\cdot}_{m}=\sqrt{\ip{\cdot}{\cdot}_{m}}\).
  In particular, from \cref{eq:massSpec,eq:stiffnesSpec} we obtain \(\norm{\psi_{m}}_{m}^{2} = d_{s}\lambda_{m}^{s}\).
  Therefore
  \begin{align*}
    \norm{U-V_{M}}_{\mathcal{H}_{\alpha}^{1}}^{2}
    &= \sum_{m=0}^{M-1}u_{m}^{2}\norm{\psi_{m}-\alpha_{m}\tilde{\psi}_{m}}_{m}^{2}  + \sum_{m=M}^{\infty}u_{m}^{2} \norm{\psi_{m}}_{m}^{2}.
  \end{align*}
  The coefficients \(\left\{\alpha_{m}\right\}\) are chosen to minimise the right-hand side.
  A simple computation reveals that the optimal choice is \(\alpha_{m}= \cos^{2} \theta_{m}\), where
  \begin{align*}
    \cos \theta_{m}
              &= \frac{\ip{\psi_{m}}{\tilde{\psi}_{m}}_{m}}{\norm{\psi_{m}}_{m} \norm{\tilde{\psi}_{m}}_{m}}
                = \sqrt{1-g\left(s, \tilde{\lambda}_{m} / \lambda_{m}\right)},
  \end{align*}
  so that
  \begin{align*}
    \norm{\psi_{m}-\alpha_{m}\tilde{\psi}_{m}}_{m}^{2}
              &= \norm{\psi_{m}}_{m}^{2}\sin^{2}\theta_{m} = d_{s}\lambda_{m}^{s}\sin^{2}\theta_{m}
              = d_{s}\lambda_{m}^{s}g\left(s, \tilde{\lambda}_{m} / \lambda_{m}\right)
  \end{align*}
  and the result follows as claimed.
\end{proof}

Observe that if the approximate eigenvalue coincides with the true eigenvalue, \(\tilde{\lambda}_{m}=\lambda_{m}\), then \(g\left(s,\tilde{\lambda}_{m} / \lambda_{m}\right)=0\) as one would expect.
By continuity, if the approximate eigenvalue is sufficiently close to the true eigenvalue, then \(g\left(s,\tilde{\lambda}_{m} / \lambda_{m}\right)\) will be small, meaning that \(\mathcal{V}_{M}\) will be a good approximation to \(\mathcal{H}_{\alpha}^{1}\left(\mathcal{C}\right)\).

The next result gives an error bound for the fully discrete scheme:

\begin{theorem}\label{thm:errorBound}
  Let \(f\in\tilde{H}^{r}\left(\Omega\right)\), for \(r\geq -s\), and choose \(M\) sufficiently large such that \(\lambda_{M}^{-(r+s)/2}\sim h^{\min\{k,r+s\}}\).
  Assume that for \(0\leq m\leq M-1\) it holds that
  \begin{align}
    g\left(s,  \tilde{\lambda}_{m} / \lambda_{m}\right)\leq \lambda_{m}^{r+s}h^{2\min\{k,r+s\}} \label{eq:eigenvalues}
  \end{align}
  and that
  \begin{align}
    \left(\frac{\tilde{\lambda}_{m}}{\lambda_{m}}\right)^{s}, \left(\frac{\lambda_{m}}{\tilde{\lambda}_{m}}\right)^{1-s} &\leq c_{\sigma}^{2} \label{eq:eigenvalues2}
  \end{align}
  with a positive constant \(c_{\sigma}\) that is independent of \(h\).
  Moreover, assume that there exist positive constants \(C_{0}\), \(C_{1}\) independent of \(h\) such that the following two inequalities hold for any \(\vec{\gamma}\in\mathbb{R}^{M}\):
  \begin{align}
    \sum_{m,n=0}^{M-1}\gamma_{m}\gamma_{n} \int_{\Omega} \left(\phi_{m}-\pi_{h}\phi_{m}\right)\left(\phi_{n}-\pi_{h}\phi_{n}\right) &\leq C_{0} \log(\lambda_{M}) \sum_{m=0}^{M-1}\gamma_{m}^{2}\norm{\phi_{m}-\pi_{h}\phi_{m}}_{L^{2}}^{2}, \label{eq:sCSI}\\
    \sum_{m,n=0}^{M-1}\gamma_{m}\gamma_{n} \int_{\Omega} \grad \left(\phi_{m}-\pi_{h}\phi_{m}\right) \cdot \grad \left(\phi_{n}-\pi_{h}\phi_{n}\right) &\leq C_{1} \log(\lambda_{M})  \sum_{m=0}^{M-1}\gamma_{m}^{2}\norm{\grad\left(\phi_{m}-\pi_{h}\phi_{m}\right)}_{L^{2}}^{2}, \label{eq:sCSIgrad}
  \end{align}
  where \(\pi_{h}\) is the Scott-Zhang interpolant \cite{ScottZhang1990_FiniteElementInterpolationNonsmooth}.
  Then, the solution \(U_{h,M}\) to the discretized extruded problem \eqref{eq:varSubspace} satisfies
  \begin{align*}
    \norm{U-U_{h,M}}_{\mathcal{H}^{1}_{\alpha}} & \leq C \abs{f}_{\tilde{H}^{r}}h^{\min\{k,r+s\}}\sqrt{\abs{\log h}},
  \end{align*}
  where \(C\) is independent of \(h\).
\end{theorem}

\begin{proof}
  By C\'ea's Lemma, the discretization error is bounded by
  \begin{align*}
    \norm{U-U_{h,M}}_{\mathcal{H}^{1}_{\alpha}}&\leq C\inf_{V_{h,M}\in\mathcal{V}_{h,M}} \norm{U-V_{h,M}}_{\mathcal{H}^{1}_{\alpha}}.
  \end{align*}
  By analogy with the proof of \Cref{lem:semianalyticBound}, we choose \(V_{h,M}\in\mathcal{V}_{h,M}\) to be
  \begin{align*}
    V_{h,M}&=\sum_{m=0}^{M-1}\alpha_{m}u_{m}\left(\pi_{h}\phi_{m}\right)\left(\vec{x}\right)\tilde{\psi}_{m}\left(y\right),
  \end{align*}
  where \(\alpha_{m}=\cos \theta_{m}\) and \(\pi_{h}\) is the Scott-Zhang interpolant \cite{ScottZhang1990_FiniteElementInterpolationNonsmooth}.
  The triangle inequality gives
  \begin{align*}
    \norm{U-V_{h,M}}_{\mathcal{H}^{1}_{\alpha}} &\leq \norm{U-V_{M}}_{\mathcal{H}^{1}_{\alpha}} + \norm{V_{M}-V_{h,M}}_{\mathcal{H}^{1}_{\alpha}}.
  \end{align*}
  The first term is easily estimated thanks to \Cref{lem:semianalyticBound,eq:eigenvalues}:
  \begin{align}
    \norm{U-V_{M}}_{\mathcal{H}^{1}_{\alpha}}^{2}
    &= d_{s}\sum_{m=0}^{M-1}u_{m}^{2}\lambda_{m}^{s}g\left(s,\tilde{\lambda}_{m} / \lambda_{m}\right) + d_{s}\sum_{m=M}^{\infty}u_{m}^{2}\lambda_{m}^{s} \nonumber \\
    &\leq d_{s}h^{2\min\{k,r+s\}}\sum_{m=0}^{M-1}u_{m}^{2}\lambda_{m}^{r+2s} + d_{s} \lambda_{M}^{-(r+s)}  \sum_{m=M}^{\infty} u_{m}^{2} \lambda_{m}^{r+2s} \nonumber \\
    &\leq d_{s}h^{2\min\{k,r+s\}} \abs{u}_{\tilde{H}^{r+2s}}^{2}, \label{eq:spectralBound}
  \end{align}
  where we recall \(M\) is chosen large enough such that \(\lambda_{M}^{-(r+s)/2}\sim h^{\min\{k,r+s\}}\).

  Turning to the second term, elementary manipulation gives
  \begin{align*}
    &\norm{V_{M}-V_{h,M}}_{\mathcal{H}^{1}_{\alpha}}^{2}\\
    =& \sum_{m=0}^{M-1}\sum_{n=0}^{M-1} \alpha_{m}\alpha_{n}u_{m}u_{n}\int_{\mathcal{C}}y^{\alpha} \grad\left[\left(\phi_{m}-\pi_{h}\phi_{m}\right)\tilde{\psi}_{m}\right]\cdot\grad\left[\left(\phi_{n}-\pi_{h}\phi_{n}\right)\tilde{\psi}_{n}\right]\\
    =& \sum_{m=0}^{M-1}\sum_{n=0}^{M-1} \alpha_{m}\alpha_{n}u_{m}u_{n} \left\{\int_{\Omega}\grad\left(\phi_{m}-\pi_{h}\phi_{m}\right)\cdot\grad\left(\phi_{n}-\pi_{h}\phi_{n}\right) \int_{0}^{\infty}y^{\alpha}\tilde{\psi}_{m}\tilde{\psi}_{n} \right.\\
    &\qquad\left. + \int_{\Omega}\left(\phi_{m}-\pi_{h}\phi_{m}\right) \left(\phi_{n}-\pi_{h}\phi_{n}\right) \int_{0}^{\infty}y^{\alpha}\tilde{\psi}_{m}'\tilde{\psi}_{n}'\right\} \\
    \leq& \sum_{m=0}^{M-1}\sum_{n=0}^{M-1} \alpha_{m}\alpha_{n}u_{m}u_{n} \left\{\int_{\Omega}\grad\left(\phi_{m}-\pi_{h}\phi_{m}\right)\cdot\grad\left(\phi_{n}-\pi_{h}\phi_{n}\right) \sqrt{\int_{0}^{\infty}y^{\alpha}\tilde{\psi}_{m}^{2}}\sqrt{\int_{0}^{\infty}y^{\alpha}\tilde{\psi}_{n}^{2}} \right.\\
    &\qquad\left. + \int_{\Omega}\left(\phi_{m}-\pi_{h}\phi_{m}\right) \left(\phi_{n}-\pi_{h}\phi_{n}\right) \sqrt{\int_{0}^{\infty}y^{\alpha}\left(\tilde{\psi}_{m}'\right)^{2}}\sqrt{\int_{0}^{\infty}y^{\alpha}\left(\tilde{\psi}_{n}'\right)^{2}}\right\} \\
    \leq& \log(\lambda_{M}) \sum_{m=0}^{M-1} \alpha_{m}^{2}u_{m}^{2} \left\{
      C_{1}\norm{\grad\phi_{m}-\grad\pi_{h}\phi_{m}}_{L^{2}}^{2} \int_{0}^{\infty}y^{\alpha}\tilde{\psi}_{m}^{2}
      + C_{0}\norm{\phi_{m}-\pi_{h}\phi_{m}}_{L^{2}}^{2} \int_{0}^{\infty}y^{\alpha}\left(\tilde{\psi}_{m}'\right)^{2}
      \right\}\\
    \leq& \max\{C_{0},C_{1}\} \log(\lambda_{M}) \sum_{m=0}^{M-1} u_{m}^{2} \left\{
      \norm{\grad\phi_{m}-\grad\pi_{h}\phi_{m}}_{L^{2}}^{2} \int_{0}^{\infty}y^{\alpha}\tilde{\psi}_{m}^{2}
      + \norm{\phi_{m}-\pi_{h}\phi_{m}}_{L^{2}}^{2} \int_{0}^{\infty}y^{\alpha}\left(\tilde{\psi}_{m}'\right)^{2}
      \right\},
  \end{align*}
  where we used \eqref{eq:sCSI}, \eqref{eq:sCSIgrad}, and that \(\alpha_{m}^{2}\leq 1\).
  Standard properties of the Scott-Zhang interpolant give
  \begin{align*}
    \norm{\grad\phi_{m}-\grad\pi_{h}\phi_{m}}_{L^{2}}&\leq Ch^{k}\abs{\phi_{m}}_{H^{k+1}}\leq Ch^{k}\lambda_{m}^{(k+1)/2},\\
    \norm{\phi_{m}-\pi_{h}\phi_{m}}_{L^{2}}&\leq Ch^{k+1}\abs{\phi_{m}}_{H^{k+1}}\leq Ch^{k+1}\lambda_{m}^{(k+1)/2},
  \end{align*}
  while, from \cref{eq:massSpec,eq:stiffnesSpec},
  \begin{align*}
    \int_{0}^{\infty}y^{\alpha}\tilde{\psi}_{m}^{2} &= d_{s}s\tilde{\lambda}_{m}^{s-1}, &\text{and} &&
    \int_{0}^{\infty}y^{\alpha}\left(\tilde{\psi}_{m}'\right)^{2} &= d_{s}(1-s)\tilde{\lambda}_{m}^{s}.
  \end{align*}
  Hence,
  \begin{align*}
    \norm{V_{M}-V_{h,M}}_{\mathcal{H}^{1}_{\alpha}}^{2}
    &\leq C \log(\lambda_{M}) h^{2k} \sum_{m=0}^{M-1} u_{m}^{2}\lambda_{m}^{k+1}\tilde{\lambda}_{m}^{s-1} + C \log(M) h^{2k+2}\sum_{m=0}^{M-1}u_{m}^{2}\lambda_{m}^{k+1}\tilde{\lambda}_{m}^{s}\\
    &\leq C \abs{u}_{\tilde{H}^{r+2s}}^{2} \abs{\log h} \left[h^{2k} \max_{m=0,\dots,M-1}\lambda_{m}^{k-(r+s)}\left(\frac{\lambda_{m}}{\tilde{\lambda}_{m}}\right)^{1-s} \right. \\
      &\qquad \left.+ h^{2k+2}\max_{m=0,\dots,M-1}\lambda_{m}^{k+1-(r+s)}\left(\frac{\tilde{\lambda}_{m}}{\lambda_{m}}\right)^{s}\right],
  \end{align*}
  where we used the fact that \(\log(\lambda_{M}) \sim \abs{\log h}\).
  Thanks to assumption \eqref{eq:eigenvalues2}, we obtain
  \begin{align*}
    \norm{V_{M}-V_{h,M}}_{\mathcal{H}^{1}_{\alpha}}^{2}
    &\leq C \abs{u}_{\tilde{H}^{r+2s}}^{2} \abs{\log h}
      \begin{cases}
        \lambda_{M-1}^{-(r+s)}\left(h^{2k}\lambda_{M-1}^{k}+h^{2k+2}\lambda_{M-1}^{k+1}\right) & \text{if } 0\leq r+s\leq k, \\
        h^{2k}\left(1+h^{2}\lambda_{M-1}^{k+1-(r+s)}\right) & \text{if } k\leq r+s \leq k+1, \\
        h^{2k}& \text{if } r+s\geq k+1,
      \end{cases}
  \end{align*}
  Recalling that \(M\) is chosen such that \(\lambda_{M}^{-(r+s)/2}\sim h^{\min\{k,r+s\}}\), we obtain
  \begin{align}
    \norm{V_{M}-V_{h,M}}_{\mathcal{H}^{1}_{\alpha}}
    &\leq C\abs{u}_{\tilde{H}^{r+2s}}h^{\min\{k,r+s\}}\sqrt{\abs{\log h}}. \label{eq:feBound}
  \end{align}

  Finally, by combining \cref{eq:spectralBound,eq:feBound}, we deduce that
  \begin{align*}
    \norm{U-U_{h,M}}_{\mathcal{H}^{1}_{\alpha}}
    &\leq C\norm{U-V_{h,M}}_{\mathcal{H}^{1}_{\alpha}}\\
    &\leq C\left(\norm{U-V_{M}}_{\mathcal{H}^{1}_{\alpha}}+\norm{V_{M}-V_{h,M}}_{\mathcal{H}^{1}_{\alpha}}\right) \\
    &\leq C\abs{u}_{\tilde{H}^{r+2s}} h^{\min\{k,r+s\}}\sqrt{\abs{\log h}}\\
    &= C\abs{f}_{\tilde{H}^{r}} h^{\min\{k,r+s\}}\sqrt{\abs{\log h}},
  \end{align*}
  since \(\abs{u}_{\tilde{H}^{r+2s}}=\abs{f}_{\tilde{H}^{r}}\).
\end{proof}

\Cref{thm:errorBound} contains two types of assumption.
Assumptions \eqref{eq:eigenvalues} and \eqref{eq:eigenvalues2} concern the approximation of the exact eigenvalues \(\left\{\lambda_{m}\right\}\) by \(\left\{\tilde{\lambda}_{m}\right\}\), which will discussed in the next section.
On the other hand, assumptions \eqref{eq:sCSI} and \eqref{eq:sCSIgrad} concern the \emph{orthogonality} of the finite element approximation error of the eigenfunctions.
In lieu of the absence of a proof of the validity of \eqref{eq:sCSI} and \eqref{eq:sCSIgrad} in general, we provide a justification in the cases where \(\Omega\) is either an interval on the real line, or the unit disc in the plane.

\emph{Example 1}: Suppose \(\Omega\) is the unit interval, \(\mathcal{T}_{h}\) is a uniform mesh with nodes \(x_{j}=jh\), \(j=0,\dots,n\), and \(\left\{\Phi_{j}\right\}\) are the piecewise linear Lagrange basis functions.
The Scott-Zhang interpolant of the eigenfunction \(\phi_{m}\left(x\right)=\frac{1}{\sqrt{\pi}}\sin\left(m \pi x\right)\) of the integer order Laplacian is given by \(\pi_{h}\phi_{m}= \vec{c}_{m}^{h}\cdot \vec{\Phi}\), where \(\vec{\Phi}\) is the vector of finite element basis functions, and \(\vec{c}_{m}^{h}=\left\{\phi_{m}\left(x_{j}\right)\right\}_{j=0}^{n}=\left\{\frac{1}{\sqrt{\pi}}\sin\left(m\pi x_{j}\right)\right\}_{j=0}^{n}\) is the finite element coefficient vector.
Moreover, the \(L^{2}\)-projection of \(\phi_{m}\) onto the space of piecewise linear functions \(V_{h}\) is given by \(\vec{c}_{m}^{L^{2}}\cdot \vec{\Phi}\), with coefficient vector \(\vec{c}_{m}^{L^{2}}=\left\{\int_{0}^{1}\phi_{m}\left(x\right)\Phi_{j}\left(x\right)\right\}_{j=0}^{n}\).
Hence, the left-hand side of \eqref{eq:sCSI} can be written as
\begin{align*}
  \int_{\Omega} \left(\phi_{m}-\pi_{h}\phi_{m}\right)\left(\phi_{n}-\pi_{h}\phi_{n}\right)
  &= \delta_{mn} - \vec{c}_{m}^{h}\cdot\vec{c}_{n}^{L^{2}} - \vec{c}_{m}^{L^{2}}\cdot\vec{c}_{n}^{h} + \vec{c}_{m}^{h}\cdot \mat{M}_{FE}\vec{c}_{n}^{h},
\end{align*}
where \(\mat{M}_{FE}\) is the mass matrix, and we used the orthogonality of the eigenfunctions \(\phi_{m}\) and \(\phi_{n}\).
Now
\begin{align*}
  \int_{0}^{1}\phi_{m}\left(x\right)\Phi_{j}\left(x\right) &= \frac{1}{\sqrt{\pi}} \sin\left(m\pi x_{j}\right) \frac{2-2\cos\left(m\pi h\right)}{\pi^{2}hm^{2}},
\end{align*}
and it follows that \(\vec{c}_{m}^{h}\) and \(\vec{c}_{m}^{L^{2}}\) are collinear.
Moreover, we recognise that \(\vec{c}_{m}^{h}\) are in fact the orthogonal eigenvectors of the tridiagonal mass matrix.
Therefore we have shown that
\begin{align*}
  \int_{\Omega} \left(\phi_{m}-\pi_{h}\phi_{m}\right)\left(\phi_{n}-\pi_{h}\phi_{n}\right) &= \delta_{mn}\norm{\phi_{m}-\pi_{h}\phi_{m}}_{L^{2}}^{2},
\end{align*}
and \eqref{eq:sCSI} holds (without the factor \(\log(\lambda_{M})\)).
A similar argument applies for \eqref{eq:sCSIgrad}.

\emph{Example 2}: \(\Omega\) is the unit disc.
In this case, we verify numerically that \eqref{eq:sCSI} and \eqref{eq:sCSIgrad} hold.
In \Cref{fig:sCSI}, we plot \(\rho\left(\mat{D}_{0}^{-1}\mat{E}_{0}\right)\) and \(\rho\left(\mat{D}_{1}^{-1}\mat{E}_{1}\right)\) versus \(M\), where
\begin{align*}
  \mat{E}_{0,mn} &= \int_{\Omega} \left(\phi_{m}-\pi_{h}\phi_{m}\right)\left(\phi_{n}-\pi_{h}\phi_{n}\right), \\
  \mat{E}_{1,mn} &= \int_{\Omega} \grad\left(\phi_{m}-\pi_{h}\phi_{m}\right)\cdot\grad\left(\phi_{n}-\pi_{h}\phi_{n}\right),
\end{align*}
for \(0\leq m,n \leq M-1\) and \(\mat{D}_{0}\) and \(\mat{D}_{1}\) are the diagonals of \(\mat{E}_{0}\) and \(\mat{E}_{1}\) respectively.
In order for assumptions \eqref{eq:sCSI} and \eqref{eq:sCSIgrad} to be satisfied, it suffices to show that
\begin{align}
  \rho\left(\mat{D}_{k}^{-1}\mat{E}_{k}\right)&\leq C_{k}\log(\lambda_{M}), \quad k=1,2. \label{eq:refsCSI}
\end{align}
In \Cref{fig:sCSI} we present the numerical values of the quantities appearing in \eqref{eq:refsCSI} for a globally quasi-uniform mesh with about 4,000 vertices for \(M\in\left\{1,\dots,100\right\}\), which suggests that \eqref{eq:sCSI}--\eqref{eq:sCSIgrad} are valid for this case.

\begin{figure}
  \centering
  \includegraphics[]{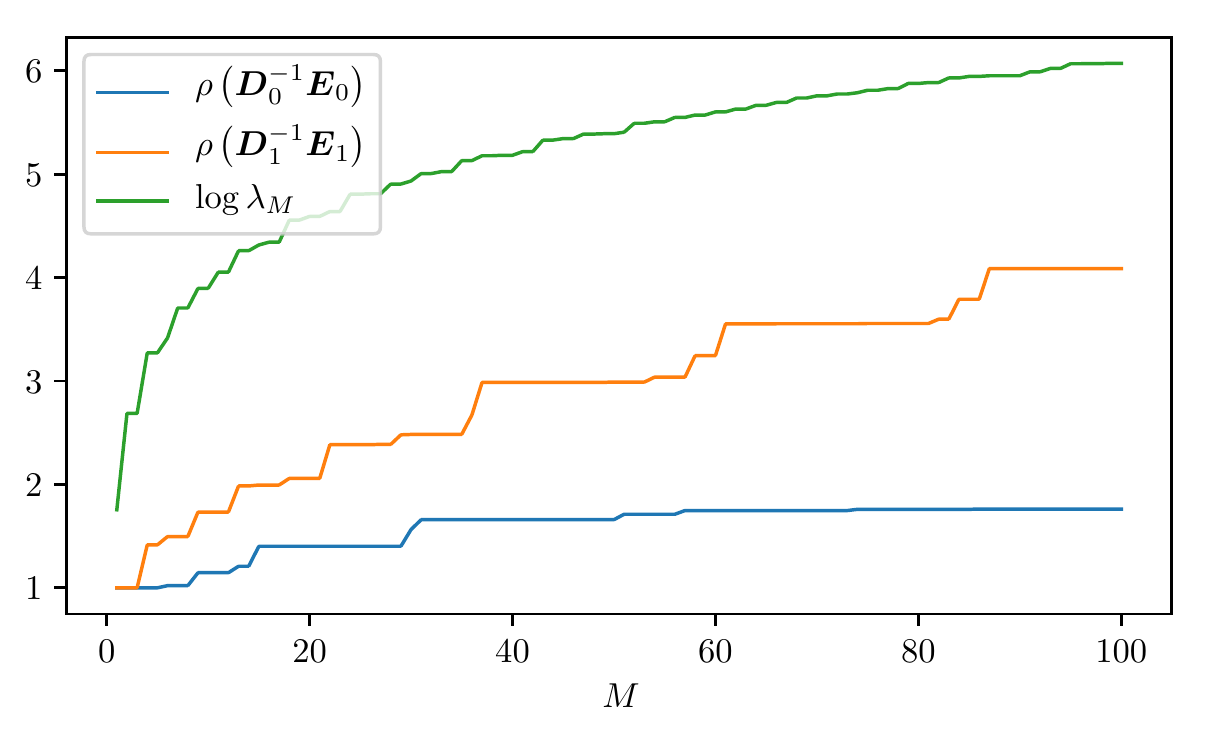}
  \caption{
    Numerical verification that \eqref{eq:sCSI} and \eqref{eq:sCSIgrad} hold in the case of the unit disc.
  }
  \label{fig:sCSI}
\end{figure}

\section{Choice of Approximate Eigenvalues \(\tilde{\lambda}_{m}\approx \lambda_{m}\)}
\label{sec:choice-appr-eigenv}

How can we find approximations \(\tilde{\lambda}_{m}\) for the eigenvalues \(\lambda_{m}\) of the standard Laplacian that satisfy conditions \eqref{eq:eigenvalues} and \eqref{eq:eigenvalues2} in \Cref{thm:errorBound} while, ideally, keeping the number of distinct approximate eigenvalues \(\tilde{M}\) as small as possible?
The following technical lemma will be useful:
\begin{lemma}\label{lem:bound}
  Let \(s\in(0,1)\), \(0\leq \varepsilon \leq \min\left\{\frac{e}{2}\frac{\min\{s,1-s\}}{\max\{s,1-s\}},1\right\}\) and \(\kappa_{s}=\sqrt{\frac{2}{e}\frac{1}{s(1-s)}}\). If
  \begin{align*}
    \abs{\log \rho} \leq \kappa_{s}\sqrt{\varepsilon}
    \quad\text{ then }\quad
    g\left(s,\rho\right)&\leq \varepsilon \text{ and } \max\{\rho^{s},\rho^{s-1}\}\leq e.
  \end{align*}
\end{lemma}
\begin{proof}
  Set \(\gamma=\log \rho\) and assume \(\abs{\gamma} \leq \kappa_{s}\sqrt{\varepsilon}\leq \frac{1}{\max\{s,1-s\}}\).
  Now, by Taylor's Theorem,
  \begin{align*}
    (1-s)\exp{s\gamma}+s\exp{(s-1)\gamma}
    &= (1-s) \left[1+s\gamma + \frac{1}{2}s^{2}\gamma^{2}\exp{s\xi}\right] \\
    &\quad + s\left[1+(s-1)\gamma + \frac{1}{2}(s-1)^{2}\gamma^{2}\exp{(s-1)\xi}\right]
  \end{align*}
  for some \(\xi\) between \(0\) and \(\gamma\) and therefore
  \begin{align*}
    (1-s)\exp{s\gamma}+s\exp{(s-1)\gamma}
    &\leq 1 + \frac{s(1-s)}{2}\gamma^{2}\exp{\max\{s,1-s\}\abs{\gamma}}\\
    &\leq 1 + \frac{s(1-s)}{2} \frac{2}{e} \frac{1}{s(1-s)}\varepsilon e \\
    &= 1+ \varepsilon\\
    &\leq \frac{1}{1-\varepsilon}.
  \end{align*}
  Hence
  \begin{align*}
    g\left(s,\rho\right)=1-\frac{1}{(1-s)\exp{s\gamma}+s\exp{(s-1)\gamma}}&\leq \varepsilon,
  \end{align*}
  and
  \begin{align*}
    \max\{\rho^{s},\rho^{s-1}\}
    &\leq \exp{\kappa_{s} \max\{s,1-s\}\sqrt{\varepsilon}}\leq e.
  \end{align*}
\end{proof}

The lemma shows that, in order to satisfy both \eqref{eq:eigenvalues} and \eqref{eq:eigenvalues2}, it suffices that the ratio \(\tilde{\lambda}_{m} / \lambda_{m}\) satisfies \(\abs{\log \tilde{\lambda}_{m} / \lambda_{m}}\leq \kappa_{s} \lambda_{m}^{(r+s)/2}h^{\min\{k,r+s\}}\).

\subsection{Approximation of Upper Part of the Spectrum - Weyl Asymptotics}
\label{sec:upper-part-spectrum}

If \(m\left(\lambda\right)\) designates the number of eigenvalues that are smaller than \(\lambda\geq 0\), then Weyl's conjecture reads
\begin{align}
  m\left(\lambda\right) &= \left(2\pi\right)^{-d}\omega_{d}\abs{\Omega}\lambda^{d/2} - \frac{1}{4}\left(2\pi\right)^{1-d}\omega_{d-1}\abs{\partial\Omega}\lambda^{(d-1)/2}+o\left(\lambda^{(d-1)/2} \right), \label{eq:WeylLaw}
\end{align}
where \(\omega_{d}=\frac{\pi^{d/2}}{\Gamma\left(1+d/2\right)}\) is the volume of the unit ball in \(\mathbb{R}^{d}\).
For more details on the exact conditions under which Weyl's law has been shown to be valid, see e.g. \cite{LiYau1983_SchroedingerEquationEigenvalueProblem,Ivrii1980_SecondTermSpectralAsymptotic}.
Neglecting all lower order terms on the right-hand side of \eqref{eq:WeylLaw} motivates the eigenvalue approximation \(\tilde{\lambda}_{m}^{\text{Weyl}}:= C_{d}\left(\frac{m}{\abs{\Omega}}\right)^{2/d}\) with \(C_{d}= 4\pi\Gamma\left(1+d/2\right)^{2/d}\).
Taking \(\lambda=\lambda_{m}\) in \eqref{eq:WeylLaw}, one obtains that \(\tilde{\lambda}_{m}^{\text{Weyl}}\) satisfies
\begin{align}
  \tilde{\lambda}_{m}^{\text{Weyl}} &= \lambda_{m}\left[1-C\lambda_{m}^{-1/2}+o\left(\lambda_{m}^{-1/2}\right)\right].\label{eq:expansionWeyl}
\end{align}
Combining \eqref{eq:expansionWeyl} with \Cref{lem:bound} shows \eqref{eq:eigenvalues2} is satisfied for sufficiently large \(\lambda_{m}\) and
\begin{align*}
  g(s,\tilde{\lambda}_{m}^{\text{Weyl}} / \lambda_{m})\leq \frac{C}{\lambda_{m}}.
\end{align*}
Therefore, the Weyl approximation \(\tilde{\lambda}_{m}^{\text{Weyl}}\) satisfies \eqref{eq:eigenvalues}, provided that \(\lambda_{m}^{-1}=\mathcal{O}\left(\lambda_{m}^{r+s}h^{2\min\{k,r+s\}}\right)\), which will be the case for all \(m\geq m_{0}\), where
\begin{align*}
  m_{0} =\mathcal{O}\left(h^{-d\min\{k,r+s\}/(1+r+s)}\right)=\mathcal{O}\left(n^{\min\{k,r+s\}/(1+r+s)}\right).
\end{align*}
We expect Weyl's conjecture to provide a good estimate for the eigenvalues in the upper part of the spectrum where \(m_{0}\leq m\leq M\).

We illustrate the approximation of the spectrum using Weyl's conjecture in the case \(\Omega=B(0,1)\subset \mathbb{R}^{2}\) (for which the exact eigenvalues \(\lambda_{m}\) are known).
\Cref{fig:bounds} shows the quantities on each side of inequality \eqref{eq:eigenvalues} for the choice \(\tilde{\lambda}_{m}=\tilde{\lambda}_{m}^{\text{Weyl}}\), where \(h\) corresponds to a quasi-uniform triangulation of the unit disc using about one million nodes.
We observe that:  the inequality \eqref{eq:eigenvalues} holds for the Weyl approximation in all but for the first few eigenvalues; and that \(g\left(s,\tilde{\lambda}_{m}^{\text{Weyl}} / \lambda_{m}\right)\) asymptotically behaves like \(\lambda_{m}^{-1}\) with only a mild variation with \(s\).
The quantity appearing on the left-hand side of the inequality \eqref{eq:eigenvalues} depends on the fractional order \(s\), and decreases as \(m\) increases.
The right-hand side, however, depends on the fractional order, the mesh size \(h\), the order \(k\) of the finite element space, and increases as \(m\) increases.
Here, we plot the right-hand side of the inequality for \(s+r\in\{0.75,1.25\}\) and \(k=1\).
We observe that as the mesh is refined, the number of eigenvalues approximated using the Weyl conjecture which fail to satisfy inequality \eqref{eq:eigenvalues} grows.

\begin{figure}
  \centering
  \includegraphics{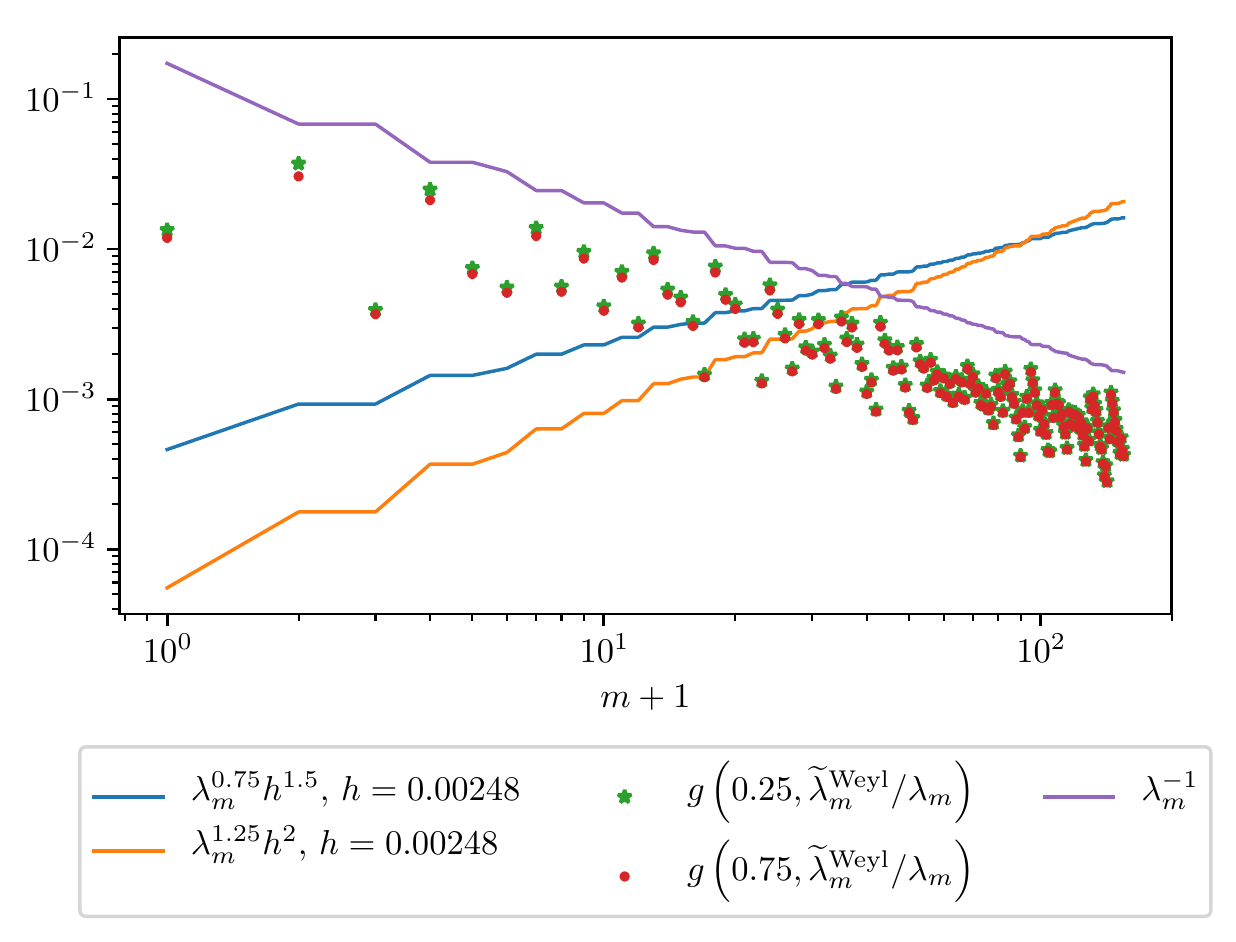}
  \caption{
    Condition \eqref{eq:eigenvalues} requires \(g\left(s, \tilde{\lambda}_{m}^{\text{Weyl}} / \lambda_{m}\right)\leq \lambda_{m}^{r+s}h^{2\min\{k,r+s\}}\).
    We display \(g\left(s, \tilde{\lambda}_{m}^{\text{Weyl}} / \lambda_{m}\right)\) and \(\lambda_{m}^{r+s}h^{2\min\{k,r+s\}}\) for \(r+s\in\{0.75, 1.25\}\) and \(k=1\).
    Here, \(h\) corresponds to a triangulation of the unit disc with about one million nodes.
    It can be observed that Weyl's conjecture gives a good approximation of the eigenvalues for the upper part of the spectrum.
  }
  \label{fig:bounds}
\end{figure}

\subsection{Finite Element Approximation of Lower Part of the Spectrum}
\label{sec:lower-part-spectrum}

The numerical example in the previous section shows that an alternative approach to Weyl's conjecture is required to approximate the smaller eigenvalues \(\lambda_{m}\), \(m=0,\dots,m_{0}\).
We propose to use the finite element method to approximate the lower part of the spectrum.
The solution of the linear system which arises in the fully discrete Galerkin problem entails the assembly of the mass matrix and the stiffness matrix for the Laplacian on the domain \(\Omega\) using finite elements which can also be used to compute approximate eigenvalues of the Laplacian.
As a matter of fact, since we are using a multigrid solver, coarser discretizations of the same problem are also readily available, meaning that we can compute eigenpairs \(\left(\tilde{\lambda}_{m,H}^{FE}, \vec{\Phi}_{m,H}\right)\) on the coarser grids:
\begin{align*}
  \mat{S}_{FE,H}\vec{\Phi}_{m,H} &=\tilde{\lambda}_{m,H}^{FE} \mat{M}_{FE,H}\vec{\Phi}_{m,H}
\end{align*}
where \(\mat{S}_{FE,H}\) and \(\mat{M}_{FE,H}\) are stiffness and mass matrix for a mesh size \(H\geq h\).

It is known that the approximate eigenvalues obtained using a finite element discretization satisfy (see e.g. \cite[Theorem 9.12]{Boffi2010_FiniteElementApproximationEigenvalueProblems} or \cite[Corollary 3.71]{ErnGuermond2004_TheoryPracticeFiniteElements})
\begin{align}
  \lambda_{m}\leq\tilde{\lambda}_{m,H}^{\text{FE}}\leq \lambda_{m} + C_{m} H^{2k}\lambda_{m}^{k+1}=\lambda_{m}\left(1 + C_{m} H^{2k}\lambda_{m}^{k}\right), \label{eq:femEigvals}
\end{align}
where \(C_{m}\) may grow as \(m\rightarrow\infty\).
In particular, if \(C_{m} H^{2k}\lambda_{m}^{k}\) is sufficiently small, then \(\log\left(1+C_{m}H^{2k}\lambda_{m}^{k}\right)\approx C_{m}H^{2k}\lambda_{m}^{k}\) is small, and, according to \Cref{lem:bound}, condition \eqref{eq:eigenvalues2} is satisfied and
\begin{align}
   g\left(s, \tilde{\lambda}_{m,H}^{FE} / \lambda_{m}\right)\leq C H^{4k}\lambda_{m}^{2k}.\label{eq:g_FE}
\end{align}
This means that \eqref{eq:eigenvalues} will be satisfied by choosing \(H\) small enough that \(H^{4k}\lambda_{m}^{2k}=\mathcal{O}\left(\lambda_{m}^{r+s}h^{2\min\{k,r+s\}}\right)\) for \(0\leq m\leq m_{0}\), or, equally well
\begin{align}
  H\leq C
  \begin{cases}
    h^{\frac{\min\{k,r+s\}}{1+r+s}\frac{1+2k}{2k}} & \text{if } 0\leq r+s\leq 2k, \\
    h^{1/2} & \text{if } r+s\geq 2k.
  \end{cases}\label{eq:H_value}
\end{align}

We illustrate these estimates by considering the case of the unit disc.
In \Cref{fig:boundsFE}, we show \(g\left(s, \tilde{\lambda}_{m,H}^{FE} / \lambda_{m}\right)\) for \(m=0,\dots,19\) obtained using several mesh sizes \(H\geq h\) and finite elements of order \(k=1\).
We also  plot the quantity appearing on the right-hand side of inequality \eqref{eq:eigenvalues}.
It can be seen that even very coarse discretizations lead to approximations that satisfy \eqref{eq:eigenvalues}.
Moreover, halving \(H\) decreases \(g\left(s,\tilde{\lambda}_{m,H}^{FE}\right)\) by a factor of 16, as suggested by \eqref{eq:g_FE}.
It can also be seen that \(g\left(s, \tilde{\lambda}_{m,H}^{FE} / \lambda_{m}\right)\) grows in proportion to \(\lambda_{m}^{2}\), as suggested by \eqref{eq:g_FE}.
The results confirm the expectation that the finite element approximation of the eigenvalues in the lower part of the spectrum provide a good choice for \(\tilde{\lambda}_{m}\).

\begin{figure}
  \centering
  \includegraphics{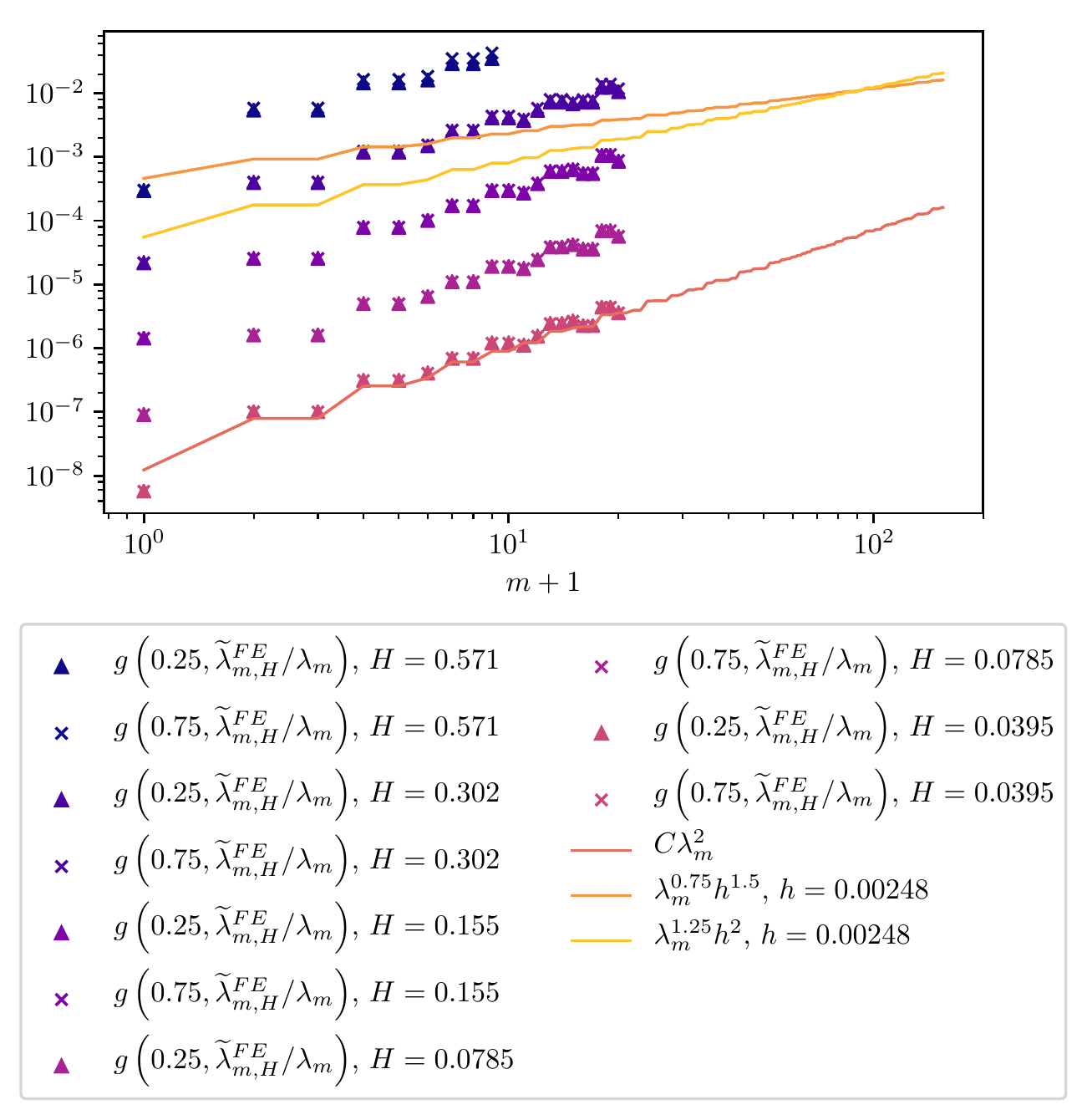}
  \caption{
    Condition \eqref{eq:eigenvalues} requires \(g\left(s, \tilde{\lambda}_{m,H}^{FE} / \lambda_{m}\right)\leq \lambda_{m}^{r+s}h^{2\min\{k,r+s\}}\).
    We display \(g\left(s, \tilde{\lambda}_{m,H}^{FE} / \lambda_{m}\right)\) for several choices of coarsened mesh sizes \(H\) against \(\lambda_{m}^{r+s}h^{2\min\{k,r+s\}}\) for \(r+s\in\{0.75, 1.25\}\) and \(k=1\).
    Here, \(h\) corresponds to a triangulation of the unit disc with about one million nodes.
    It can be observed that the finite element approximation of the eigenvalues provides a good choice for the approximation of the lower part of the spectrum.
    }
  \label{fig:boundsFE}
\end{figure}




\subsection{Size Reduction of the Approximation Space}
\label{sec:size-reduct-appr}

Suppose we have a candidate sequence of approximate eigenvalues \(\tilde{\lambda}_{m}\), \(m=0,\dots,M\).
These might coincide with the exact eigenvalues, if they are known, or could be obtained by a combination of finite element and Weyl approximations as described earlier.
In general, both the finite element approximations \(\tilde{\lambda}_{m}^{\text{FE}}\) and the approximations \(\tilde{\lambda}_{m}^{\text{Weyl}}\) from Weyl's law will be distinct.
This implies that the number of approximate eigenvalues is \(\tilde{M}=M\), and therefore the dimension of the approximation space \(\mathcal{V}_{h,M}\) would be \(\mathcal{N}=nM\).
However, it is unnecessary for the approximate eigenvalues to be in one to one correspondence with the true eigenvalues.
For instance, if two true eigenvalues are close together, then a single approximate eigenvalue should suffice for both.
This effectively reduces the number of approximate eigenvalues to \(\tilde{M}\leq M\).
Accordingly, we propose to minimise the number of distinct eigenvalues \(\left\{\tilde{\lambda}_{m}\right\}_{m=0}^{\tilde{M}-1}\) whilst still satisfying the bounds of \cref{eq:eigenvalues,eq:eigenvalues2}.
Employing \Cref{lem:bound}, we select a new set of approximations \(\hat{\lambda}_{m}\) by choosing \(\hat{\lambda}_{0}=\tilde{\lambda}_{0}\), and for \(m\geq1\),
\begin{align}
  \hat{\lambda}_{m}&=
                     \begin{cases}
                       \hat{\lambda}_{m-1} & \text{if } \abs{\log \frac{\hat{\lambda}_{m-1}}{\tilde{\lambda}_{m}}} \leq \kappa_{s} \min\left\{\left(\tilde{\lambda}_{m}^{\text{Weyl}}\right)^{(r+s)/2}h^{\min\{k,r+s\}},\sqrt{\frac{e}{2}\frac{\min\{s,1-s\}}{\max\{s,1-s\}}},1\right\}\\
                       \tilde{\lambda}_{m} &\text{otherwise}
                       \end{cases}. \label{eq:decimation}
\end{align}
Here, we have used the fact that the Weyl approximations \(\tilde{\lambda}_{m}^{\text{Weyl}}\) bound the exact eigenvalues from below.
We will see in the numerical examples in \Cref{sec:numerical-examples} (e.g. \Cref{fig:space_dimension_disc}) that this procedure results in \(\tilde{M} \ll M\).

To illustrate the method, we again consider the case where the domain is chosen to be the unit disc.
In \Cref{fig:boundsHybrid}, we display \(g\left(s, \hat{\lambda}_{m} / \lambda_{m}\right)\), where \(\hat{\lambda}_{m}\) is obtained by collapsing eigenvalue approximations obtained through finite element approximation and Weyl's law as described above.
We observe that \eqref{eq:eigenvalues} remains valid.

\begin{figure}
  \centering
  \includegraphics{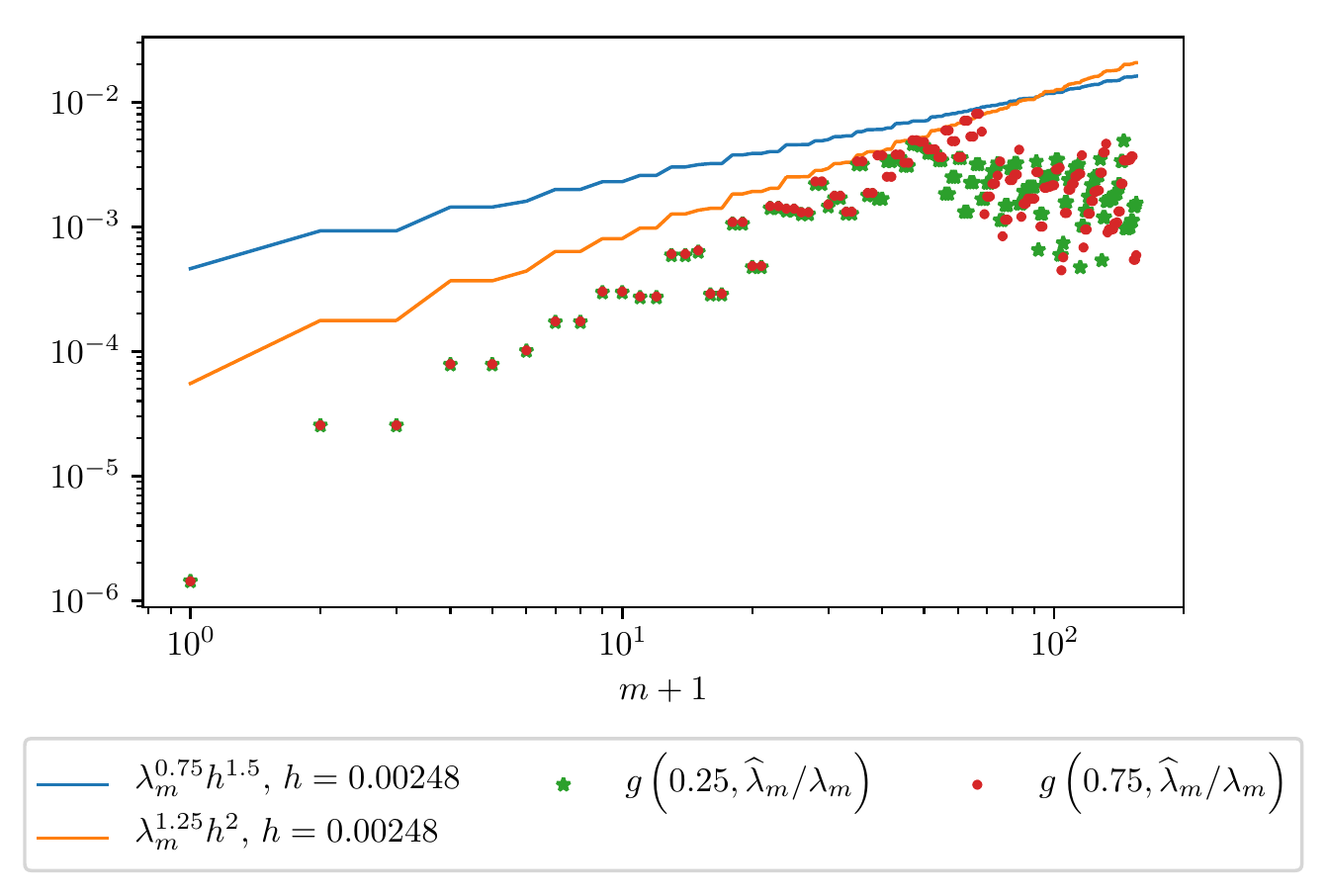}
  \caption{
    Condition \eqref{eq:eigenvalues} requires \(g\left(s, \hat{\lambda}_{m} / \lambda_{m}\right)\leq \lambda_{m}^{r+s}h^{2\min\{k,r+s\}}\).
    We display \(g\left(s, \hat{\lambda}_{m} / \lambda_{m}\right)\) and \(\lambda_{m}^{r+s}h^{2\min\{k,r+s\}}\) for \(r+s\in\{0.75, 1.25\}\) and \(k=1\).
    Here, \(h\) corresponds to a triangulation of the unit disc with about one million nodes.
    We observe that \eqref{eq:eigenvalues} is still satisfied.
    }
  \label{fig:boundsHybrid}
\end{figure}

\section{Solution of the Linear Algebraic System}
\label{sec:solution-linear-system}

Let \(\left\{\Phi_{i}\right\}_{i=1}^{n}\) denote the nodal basis functions of the finite element solution space \(V_{h}\), then the solution of the discretized fractional Poisson problem can be written as \(u_{h,M}\left(\vec{x}\right)=\sum_{i=1}^{n}d_{i}\Phi_{i}\left(\vec{x}\right)\).

Here, for ease of notation, we assume that the eigenvalue approximations \(\tilde{\lambda}_{m}\), \(m=0,\dots,\tilde{M}-1\), are all distinct.
Obviously, this can easily be achieved by relabelling the reduced set of eigenvalues resulting from the procedure described by \eqref{eq:decimation}.

The solution of the extruded problem \eqref{eq:varSubspace} can be written in the form
\begin{align*}
  U_{h,M}\left(\vec{x},y\right)=\sum_{m=0}^{\tilde{M}-1}\sum_{i=1}^{n}c_{i,m}\Phi_{i}\left(\vec{x}\right)\tilde{\psi}_{m}\left(y\right)\in\mathcal{V}_{h,M}
\end{align*}
with the coefficients \(\left(c_{i,m}\right)=\vec{U}_{h,M}\) obtained by solving the linear system
\begin{align}
  \left(\mat{M}_{FE}\otimes \mat{S}_{\sigma}+ \mat{S}_{FE}\otimes \mat{M}_{\sigma}\right)\vec{U}_{h,M}=\vec{F}_{h,M},\label{eq:linearSystem}
\end{align}
where
\begin{align*}
  \mat{M}_{FE} &= \left(\int_{\Omega}\Phi_{i}\Phi_{j}\right), &
  \mat{S}_{FE} &= \left(\int_{\Omega}\grad\Phi_{i}\grad\Phi_{j}\right), \\
  \mat{M}_{\sigma} &= \left(\int_{0}^{\infty}y^{\alpha}\tilde{\psi}_{m}\tilde{\psi}_{n}\right), &
  \mat{S}_{\sigma} &= \left(\int_{0}^{\infty}y^{\alpha}\tilde{\psi}_{m}'\tilde{\psi}_{n}'\right), \\
  \vec{F}_{h,M} &= \vec{f}_{h} \otimes \vec{1}_{\tilde{M}},&
  \vec{f}_{h}&= \left(d_{s} \pair{f_{h}}{\Phi_{i}}\right).
\end{align*}
Here, \(\vec{1}_{\tilde{M}}\) is the vector of ones of length \(\tilde{M}\).
The approximation to the solution of the fractional Poisson problem is then obtained by taking the trace of \(U_{h,M}\) on \(\Omega\):
\begin{align}
  u_{h,M}
  &=\operatorname{tr}_{\Omega}U_{h,M}
  = \sum_{i=1}^{n}\left(\sum_{m=0}^{\tilde{M}-1}c_{i,m}\right)\Phi_{i}\left(\vec{x}\right),\label{eq:discSol}
\end{align}
where we recall the normalisation \(\tilde{\psi}_{m}\left(0\right)=1\).
In matrix form, the trace operator is given by \(\mat{I}\otimes\vec{1}_{\tilde{M}}^{T}\in\mathbb{R}^{n\times \mathcal{N}}\), so that \(\vec{u}_{h,M}=\left[\mat{I}\otimes\vec{1}_{\tilde{M}}^{T}\right]\vec{U}_{h,M}\).

\Cref{eq:massSpec,eq:stiffnesSpec} show that both the spectral mass and stiffness matrices are symmetric and dense.
In order to compute the solution of \eqref{eq:linearSystem} efficiently, we consider the Cholesky factorisation of \(\mat{M}_{\sigma}\):
\begin{align*}
  \mat{M}_{\sigma}=\mat{L}\mat{L}^{T}
\end{align*}
where \(\mat{L}\) is lower triangular; and the eigenvalue decomposition of \(\mat{L}^{-1}\mat{S}_{\sigma}\mat{L}^{-T}\)
\begin{align*}
  \mat{S}_{\sigma}&=\mat{L}\mat{P}\mat{\Lambda}\mat{P}^{T}\mat{L}^{T}
\end{align*}
where \(\mat{\Lambda}\) is diagonal and \(\mat{P}\) is orthogonal.
Each factorisation can be computed in \(\mathcal{O}\left(\tilde{M}^{3}\right)\) operations.
These factorisations allow the matrix appearing in \eqref{eq:linearSystem} to be factorised as
\begin{align*}
  \mat{M}_{FE}\otimes \mat{S}_{\sigma}+ \mat{S}_{FE}\otimes \mat{M}_{\sigma}
  &= \left[ \mat{I}\otimes \left(\mat{L}\mat{P}\right)\right] \left[(\mat{M}_{FE}\otimes \mat{\Lambda} + \mat{S}_{FE}\otimes \mat{I}\right]\left[\mat{I}\otimes \left(\mat{L}\mat{P}\right)^{T}\right] .
\end{align*}
with the inverse given by
\begin{align*}
  \left(\mat{M}_{FE}\otimes \mat{S}_{\sigma}+ \mat{S}_{FE}\otimes \mat{M}_{\sigma}\right)^{-1}
  &= \left[\mat{I}\otimes \left(\mat{L}^{-T}\mat{P}\right)\right] \left[\mat{M}_{FE}\otimes \mat{\Lambda} + \mat{S}_{FE}\otimes \mat{I}\right]^{-1} \left[\mat{I}\otimes \left(\mat{P}^{T}\mat{L}^{-1}\right)\right].
\end{align*}
Using this form of the inverse to write down an explicit expression for the solution \(\vec{U}_{h,M}\) of \eqref{eq:linearSystem} and then inserting into \eqref{eq:discSol}, taking account of the right-hand side and applying the discrete trace operator gives
\begin{align}
  \vec{u}_{h,M}=&\left[\mat{I}\otimes \left(\mat{P}^{T}\mat{L}^{-1}\vec{1}_{\tilde{M}}\right)^{T}\right] \left[\mat{M}_{FE}\otimes \mat{\Lambda} + \mat{S}_{FE}\otimes \mat{I}\right]^{-1} \left[\vec{f}_{h}\otimes \left(\mat{P}^{T}\mat{L}^{-1}\vec{1}_{\tilde{M}}\right)\right] \nonumber \\
  =& \sum_{m=0}^{\tilde{M}-1} w_{m}^{2} \left[\mat{M}_{FE}\mat{\Lambda}_{mm} + \mat{S}_{FE}\right]^{-1} \vec{f}_{h}. \label{eq:linearSystemReduced}
\end{align}
Here, \(\vec{w}\) denotes the weight vector \(\vec{w}=\mat{P}^{T}\mat{L}^{-1}\vec{1}_{\tilde{M}}\in\mathbb{R}^{\tilde{M}}\), which can be computed in \(\mathcal{O}\left(\tilde{M}^{2}\right)\) operations and stored for reuse.
It remains to compute the action of the inverse \(\left[\mat{M}_{FE}\otimes \mat{\Lambda} + \mat{S}_{FE}\otimes \mat{I}\right]^{-1}\).
This is accomplished using a conjugate gradient solver with standard geometric multigrid preconditioner for the solution of the systems \(\mat{M}_{FE}\Lambda_{mm}+\mat{S}_{FE}\), \(m=0,\dots,\tilde{M}-1\), meaning that each system can be solved in \(\mathcal{O}\left(n\right)\) operations.

In summary, the setup of the multigrid solver, the prefactorisation of the matrices and the precomputation of \(\vec{w}\) will cost \(\mathcal{O}\left(n+\tilde{M}^{3}\right)\) operations, and each solve will cost \(\mathcal{O}\left(n\tilde{M}\right)=\mathcal{O}\left(\mathcal{N}\right)\) operations.
The parallelisation of the solution procedure can take advantage of the fact that each of the solves in \eqref{eq:linearSystemReduced} is independent.

\section{Numerical Examples}
\label{sec:numerical-examples}

\subsection{Piecewise Linear Finite Element Approximation on the Unit Disk}
\label{sec:unit-disk}

Consider the problem
\begin{align*}
  \left\{
  \begin{array}{rlrl}
    \left(-\Delta\right)^{s}u&=f && \text{in } \Omega=B(0,1)\subset \mathbb{R}^{2}\\
    u&=0 && \text{on }\partial\Omega,
  \end{array}\right.
\end{align*}
where \(f=\left(1-\abs{\vec{x}}^{2}\right)^{r-1/2}\in\tilde{H}^{r-\varepsilon}\left(\Omega\right)\), for all \(\varepsilon>0\).
We approximate the solution for \(s\in\{0.25,0.75\}\) and \(r\in\{0.5, 2\}\) using piecewise linear finite elements (i.e. \(k=1\)).

The true eigenvalues and eigenfunctions of the Laplacian
\begin{align*}
  \left\{
  \begin{array}{rlrl}
    -\Delta\phi_{k,\ell}&=\lambda_{k,\ell}\phi_{k,\ell} && \text{in } \Omega\\
    u&=0 && \text{on }\partial\Omega,
  \end{array}\right.
\end{align*}
are given by
\begin{align*}
  \phi_{k,0}&=\frac{1}{\sqrt{\pi}J_{1}\left(\alpha_{0,k}\right)}J_{0}\left(\alpha_{0,k}r\right), && k\geq 1, \\
  \phi_{k,\ell}&=\frac{\sqrt{2}}{\sqrt{\pi}J_{\ell+1}\left(\alpha_{\ell,k}\right)} J_{\ell}\left(\alpha_{\ell,k}r\right) \cos\left(\ell\theta\right), && k\geq 1, \ell\geq 1,\\
  \phi_{k,-\ell}&=\frac{\sqrt{2}}{\sqrt{\pi}J_{\ell+1}\left(\alpha_{\ell,k}\right)} J_{\ell}\left(\alpha_{\ell,k}r\right) \sin\left(\ell\theta\right),&& k\geq 1, \ell\geq 1, \\
  \lambda_{k,\ell}&=\lambda_{k,-\ell}=\alpha_{\ell,k}^{2},
\end{align*}
where \(J_{\ell}\) are the Bessel functions of the first kind and \(\alpha_{\ell,k}\) are the zeros of \(J_{\ell}\).
Although the true eigenvalues are known for this case, we do not use this information in the definition of the solution space \(\mathcal{V}_{h,M}\).
Instead, we use the approximations obtained via finite elements and Weyl's law detailed above.
In order to assess the overall accuracy, we evaluate the error in the approximation using the expression:
\begin{align*}
  \norm{U-U_{h,M}}_{\mathcal{H}^{1}_{\alpha}}^{2}
  &= \norm{U}_{\mathcal{H}^{1}_{\alpha}}^{2} - 2\pair{U}{U_{h,M}}_{\mathcal{H}^{1}_{\alpha}} + \norm{U_{h,M}}_{\mathcal{H}^{1}_{\alpha}}^{2} \\
  &= \norm{U}_{\mathcal{H}^{1}_{\alpha}}^{2} - d_{s}\pair{f}{u_{h,M}}.
\end{align*}
Then, expanding the data \(f\) as a Bessel series
\begin{align*}
  f&= \sum_{k,\ell} f_{k,\ell}\phi_{k,\ell}, \quad \text{where} \quad f_{k,\ell}=\ip{f}{\phi_{k,\ell}}_{L^{2}} = \delta_{\ell,0}2^{r+1/2}\sqrt{\pi}\Gamma\left(r+1/2\right) \frac{J_{r+1/2}\left(\alpha_{0,k}\right)}{\alpha_{0,k}^{r+1/2}J_{1}\left(\alpha_{0,k}\right)},
\end{align*}
we obtain
\begin{align*}
  \norm{U-U_{h,M}}_{\mathcal{H}^{1}_{\alpha}}^{2}
  &= d_{s}\sum_{k,\ell} f_{k,\ell}^{2}\lambda_{k,\ell}^{-s} - d_{s}\pair{f}{u_{h,M}}\\
  &= d_{s} \left\{2^{2r+1}\pi\Gamma\left(r+1/2\right)^{2} \sum_{k=1}^{\infty}  \left(\frac{J_{r+1/2}\left(\alpha_{0,k}\right)}{\alpha_{0,k}^{s+r+1/2}J_{1}\left(\alpha_{0,k}\right)}\right)^{2} - \pair{f}{u_{h,M}}\right\}.
\end{align*}
In practice, we truncate the summation but keep sufficiently many terms that the error from the truncation is negligible in comparison with the error in the Galerkin scheme.

In \Cref{fig:error_h_disc}, we plot the \(\mathcal{H}^{1}_{\alpha}\)-error with respect to the mesh size \(h\).
It is observed that the error decays as predicted by \Cref{thm:errorBound}.
In \Cref{fig:error_N_disc}, we again show the \(\mathcal{H}^{1}_{\alpha}\)-error, this time with respect to the total number of degrees of freedom \(\mathcal{N}\).
Letting \(n=\dim V_{h}\), we have
\begin{align*}
  \norm{U-U_{h,M}}_{\mathcal{H}^{1}_{\alpha}}
  &\leq C\abs{f}_{\tilde{H}^{r}}h^{\min\{k,r+s\}}\sqrt{\abs{\log h}}
  \leq C\abs{f}_{\tilde{H}^{r}}n^{-\min\{k,r+s\}/d}\sqrt{\log n}.
\end{align*}
Suppose that the number of distinct eigenvalue approximations behaves like \(\tilde{M}=\mathcal{O}\left(\log^{p} n\right)\) for some \(p\geq 0\).
Then the total number of degrees of freedom is \(\mathcal{N}=n\tilde{M}=\mathcal{O}\left(n \log^{p} n\right)\).
That is to say, the total number of degrees of freedom scales like the number of degrees of freedom in the usual, integer order case, apart from the logarithmic factor.
In this case, we would obtain quasi-optimal \(\mathcal{H}^{1}_{\alpha}\)-error convergence:
\begin{align}
  \norm{U-U_{h,M}}_{\mathcal{H}^{1}_{\alpha}} \leq C \abs{f}_{\tilde{H}^{r}}\mathcal{N}^{-\min\{k,r+s\}/d} \log^{q}\mathcal{N} \label{eq:convN}
\end{align}
for some \(q\geq 0\) up to a logarithmic factor.
It is observed in \Cref{fig:error_N_disc} that this behaviour is observed in practice.
In fact, \(\tilde{M}=\mathcal{O}\left(\log^{p} n\right)\) for some exponent \(p\geq1\), as can be seen from \Cref{fig:space_dimension_disc}, and the method displays quasi-optimal complexity as observed in \Cref{fig:error_N_disc}.

In order to assess the efficiency of the solver for the linear algebraic system, in \Cref{fig:iterations_disc} we show the average number of iterations of multigrid preconditioned conjugate gradient necessary to solve the systems \(\mat{M}_{FE}+\mat{\Lambda}_{mm}\mat{S}_{FE}\), \(m=0,\dots,\tilde{M}-1\).
Observe that roughly 10 iterations are required for convergence independently of problem size, regularity of the data or fractional order.

Finally, we display timing results for setup and solution in \Cref{fig:timings_disc}.
It can be seen that both the setup time for the solver (which includes the approximation of eigenvalues) and solution of the resulting linear system of equations scale as \(\mathcal{O}\left(n\right)\), where \(n\) is the number of degrees of freedom in the finite element discretization.

\begin{figure}
  \centering
  \includegraphics{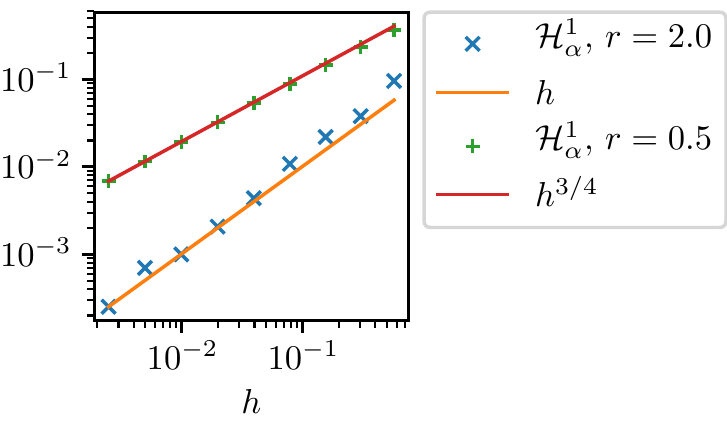}
  \includegraphics{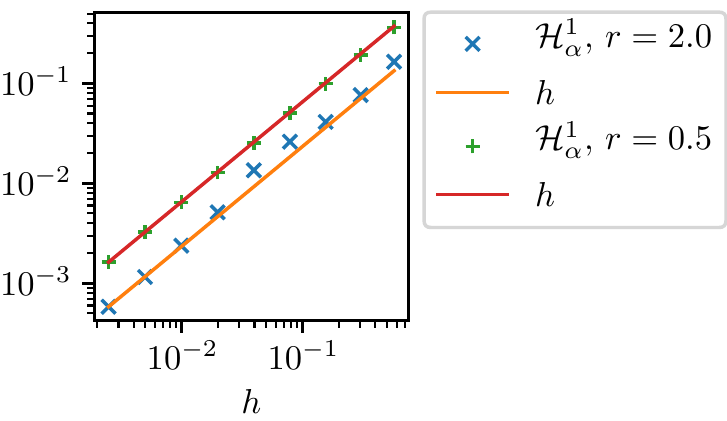}
  \caption{
    \(\mathcal{H}^{1}_{\alpha}\)-error for the fractional Poisson problem with right-hand side \(f=\left(1-\abs{\vec{x}}^{2}\right)^{r-1/2}\) on the unit disc with piecewise linear finite elements (\(k=1\)).
    \(s=0.25\) on the left, \(s=0.75\) on the right.
    The error decay of \(h^{\min\{k,r+s\}}\) predicted by \Cref{thm:errorBound} is observed.
  }
  \label{fig:error_h_disc}
\end{figure}

\begin{figure}
  \centering
  \includegraphics{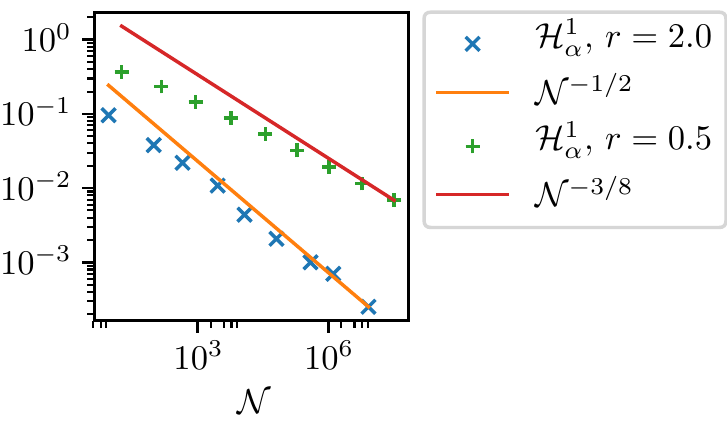}
  \includegraphics{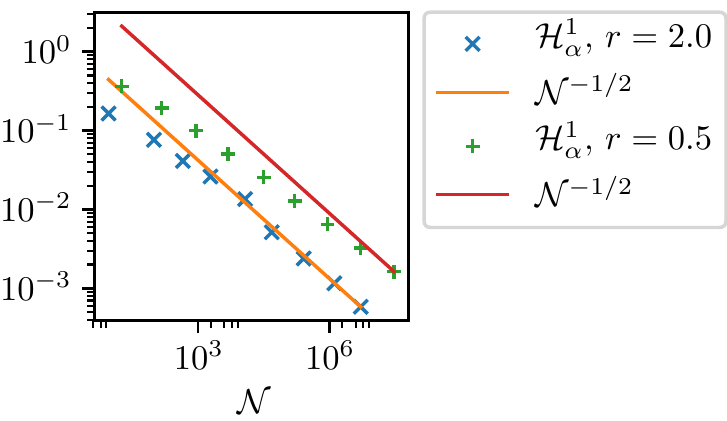}
  \caption{
    \(\mathcal{H}^{1}_{\alpha}\)-error with respect to the total number of degrees of freedom \(\mathcal{N}\) on the unit disc with piecewise linear finite elements (\(k=1\)).
    \(s=0.25\) on the left, \(s=0.75\) on the right.
    Quasi-optimal convergence is obtained.
    (Compare with the optimal order given in \eqref{eq:convN}.)
  }
  \label{fig:error_N_disc}
\end{figure}

\begin{figure}
  \centering
  \includegraphics{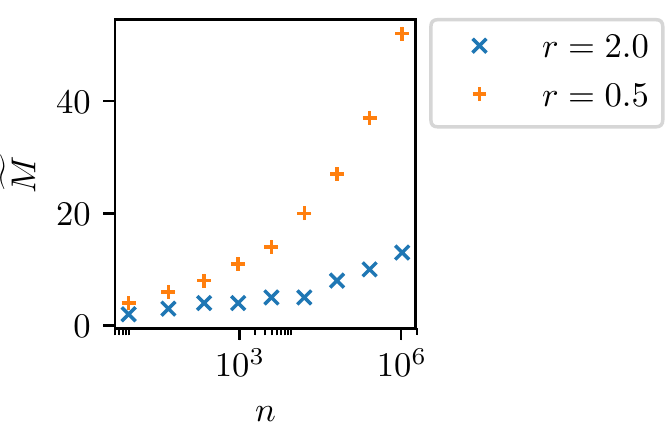}
  \includegraphics{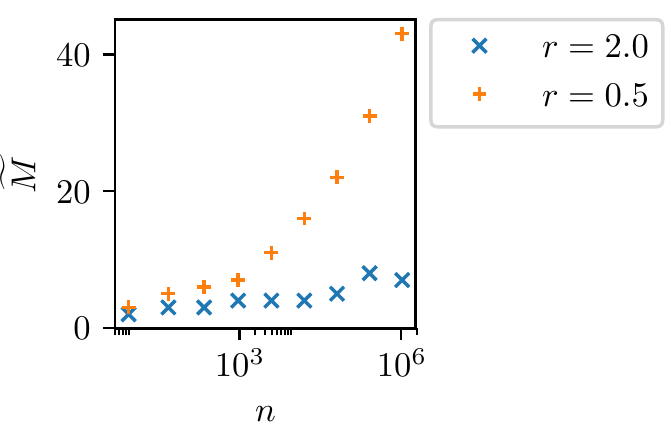}
  \caption{
    Number of distinct eigenvalues \(\tilde{M}\).
    \(s=0.25\) on the left, \(s=0.75\) on the right.
    The number of distinct eigenvalue approximations \(\tilde{M}\) grows like \(C\log^{p} n\) for some \(p\geq 1\).
  }
  \label{fig:space_dimension_disc}
\end{figure}

\begin{figure}
  \centering
  \includegraphics{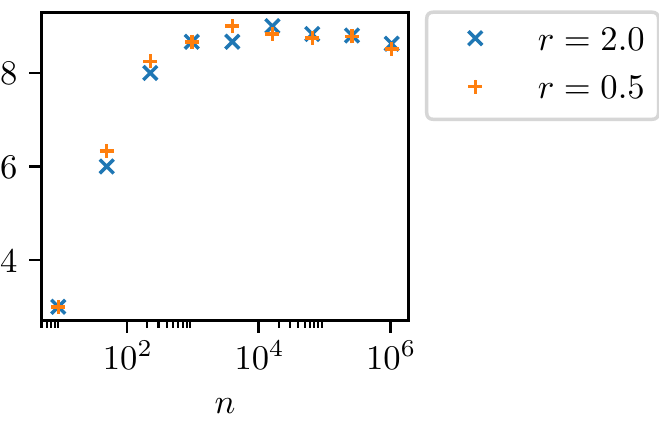}
  \includegraphics{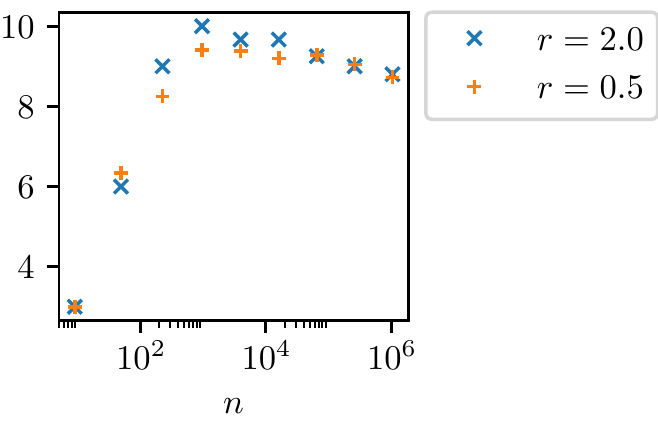}
  \caption{
    Average number of multigrid preconditioned conjugate gradient iterations.
    \(s=0.25\) on the left, \(s=0.75\) on the right.
    We observe that 10 iterations are sufficient for convergence, independent of problem size, right-hand side regularity and fractional order.
  }
  \label{fig:iterations_disc}
\end{figure}

\begin{figure}
  \centering
  \includegraphics{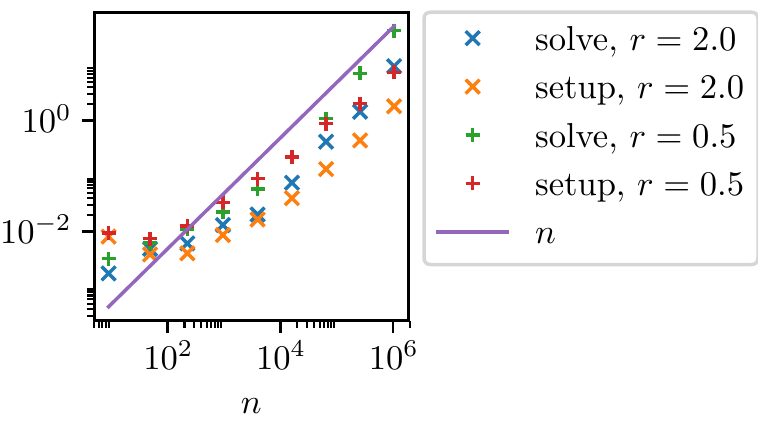}
  \includegraphics{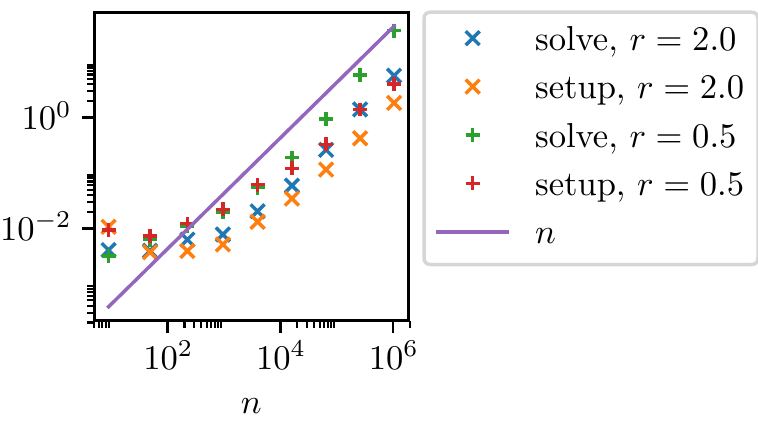}
  \caption{
    Timings of setup and solution.
    \(s=0.25\) on the left, \(s=0.75\) on the right.
    It can be seen that both setup of the solver, which includes the approximation of eigenvalues, and solution of the resulting linear system of equations scale roughly as \(\mathcal{O}\left(n\right)\), where \(n\) is the number of degrees of freedom of the finite element discretization.
  }
  \label{fig:timings_disc}
\end{figure}

\FloatBarrier{}

\subsection{Piecewise Quadratic Finite Element Approximation on the Unit Square}
\label{sec:unit-square}

Consider now the approximation of the fractional Poisson problem on the unit square:
\begin{align*}
  \left\{
  \begin{array}{rlrl}
    \left(-\Delta\right)^{s}u&=f && \text{in } \Omega=[0,1]^{2}\\
    u&=0 && \text{on }\partial\Omega,
  \end{array}\right.
\end{align*}
where \(f\left(\vec{x}\right)=\left[x_{1}x_{2}(1-x_{1})(1-x_{2})\right]^{r-1/2}\).
This time we use piecewise quadratic finite elements in order to demonstrate the flexibility of the approach.
The exact eigenvalues and eigenfunctions are known and can be used to compute the \(\mathcal{H}^{1}_{\alpha}\)-error using the expression
\begin{align*}
  &\norm{U-U_{h,M}}_{\mathcal{H}^{1}_{\alpha}}^{2}\\
  =& d_{s}\left\{4\pi^{2} \Gamma\left(r+1/2\right)^{4} \right.\\
  &\qquad \times \sum_{p,q=0}^{\infty} \frac{1}{\pi^{4r+2s}} \frac{1}{\left(2p+1\right)^{2r} \left(2q+1\right)^{2r}\left[(2p+1)^{2}+(2q+1)^{2}\right]^{s}} J_{r}\left(\pi(p+1/2)\right)^{2} J_{r}\left(\pi(q+1/2)\right)^{2} \\
  &\qquad \left.- \pair{f}{u_{h,M}}\right\}.
\end{align*}
As before, we wish to assess the convergence rate of our procedure when the eigenvalues are approximated using Weyl's law and finite element approximations for the definition of the solution space.

In \Cref{fig:error_h_square,fig:error_N_square}, we show the \(\mathcal{H}^{1}_{\alpha}\)-error versus \(h\) and \(\mathcal{N}\) respectively.
It can be seen that the error bound of \Cref{thm:errorBound} is satisfied, and that quasi-optimal convergence with respect to \(\mathcal{N}\) is again obtained.

\begin{figure}
  \centering
  \includegraphics{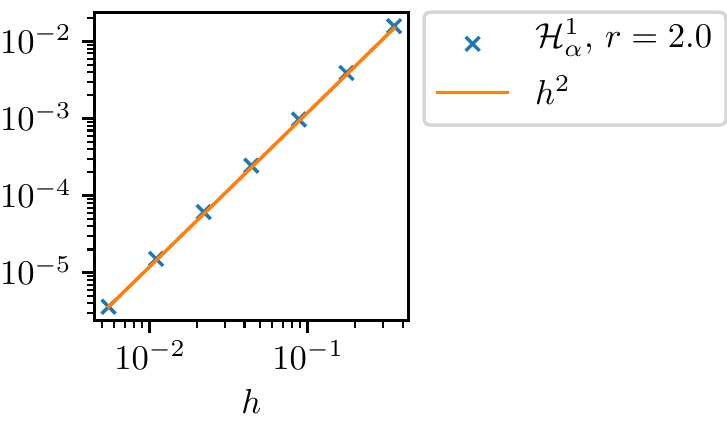}
  \includegraphics{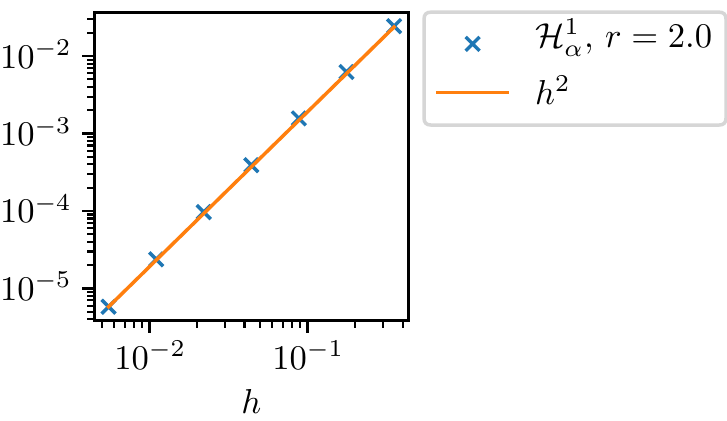}
  \caption{
    \(\mathcal{H}^{1}_{\alpha}\)-error for the fractional Poisson problem with right-hand side \(f=\left[x_{1}x_{2}(1-x_{1})(1-x_{2})\right]^{r-1/2}\) on the unit square with piecewise quadratic finite elements (\(k=2\)).
    \(s=0.25\) on the left, \(s=0.75\) on the right.
    The error decay of \(h^{\min\{k,r+s\}}\) predicted by \Cref{thm:errorBound} is observed.
  }
  \label{fig:error_h_square}
\end{figure}

\begin{figure}
  \centering
  \includegraphics{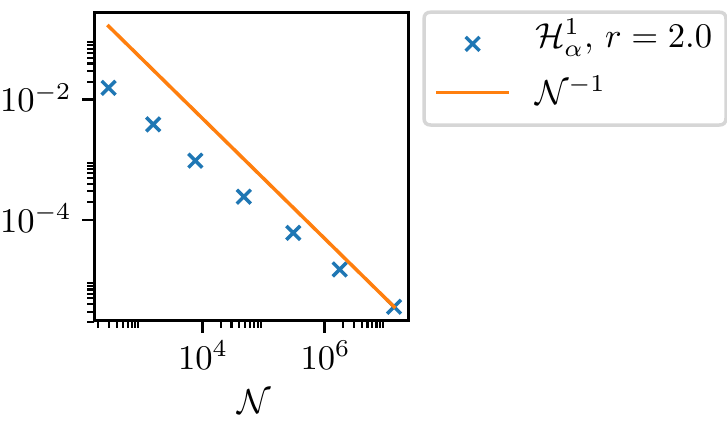}
  \includegraphics{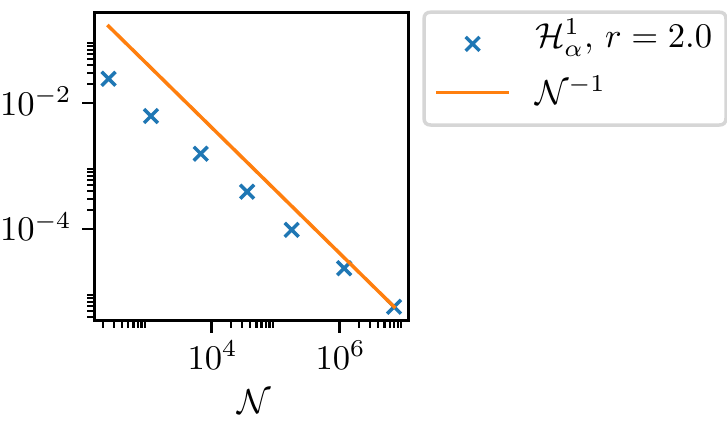}
  \caption{
    \(\mathcal{H}^{1}_{\alpha}\)-error with respect to the total number of degrees of freedom \(\mathcal{N}\) on the unit square with piecewise quadratic finite elements (\(k=2\)).
    \(s=0.25\) on the left, \(s=0.75\) on the right.
    Quasi-optimal convergence is obtained.
    (Compare with the optimal order given in \eqref{eq:convN}.)
  }
  \label{fig:error_N_square}
\end{figure}

\FloatBarrier{}

\subsection{Piecewise Linear Finite Element Approximation on the Unit Cube}
\label{sec:unit-cube}

Finally, consider a fractional Poisson problem in three dimensions on the unit cube:
\begin{align*}
  \left\{
  \begin{array}{rlrl}
    \left(-\Delta\right)^{s}u&=f && \text{in } \Omega=[0,1]^{3}\\
    u&=0 && \text{on }\partial\Omega,
  \end{array}\right.
\end{align*}
We use piecewise linear finite element approximation and compute the true \(\mathcal{H}^{1}_{\alpha}\)-error in similar fashion as before.
Here, \(f\left(\vec{x}\right)=\left[x_{1}x_{2}x_{3}(1-x_{1})(1-x_{2})(1-x_{3})\right]^{r-1/2}\).
In \Cref{fig:error_h_cube,fig:error_N_cube}, we plot the \(\mathcal{H}^{1}_{\alpha}\)-error versus \(h\) and \(\mathcal{N}\) respectively.
It can be seen that the error bound of \Cref{thm:errorBound} is satisfied, and that quasi-optimal convergence with respect to \(\mathcal{N}\) is again observed.

\begin{figure}
  \centering
  \includegraphics{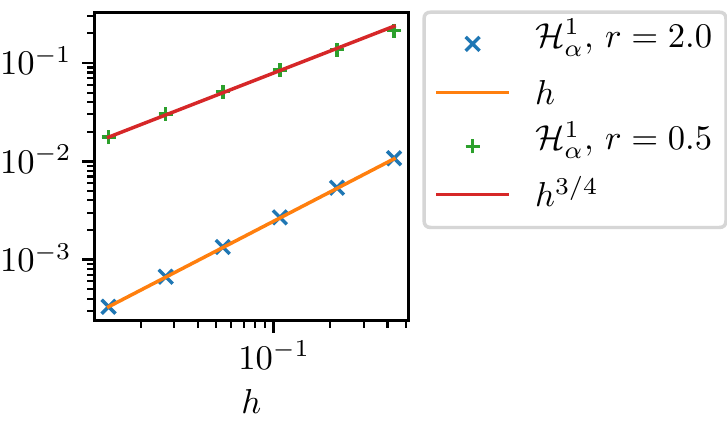}
  \includegraphics{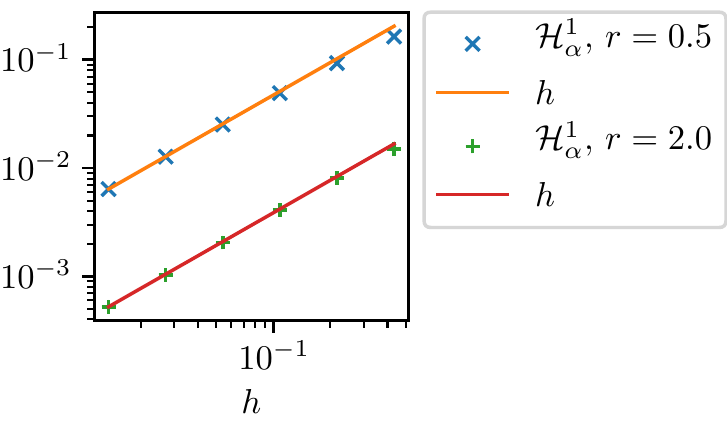}
  \caption{
    \(\mathcal{H}^{1}_{\alpha}\)-error for the fractional Poisson problem with right-hand side \(f=\left[x_{1}x_{2}x_{3}(1-x_{1})(1-x_{2})(1-x_{3})\right]^{r-1/2}\) on the unit cube with piecewise linear finite elements (\(k=1\)).
    \(s=0.25\) on the left, \(s=0.75\) on the right.
    The error decay of \(h^{\min\{k,r+s\}}\) predicted by \Cref{thm:errorBound} is observed.
  }
  \label{fig:error_h_cube}
\end{figure}

\begin{figure}
  \centering
  \includegraphics{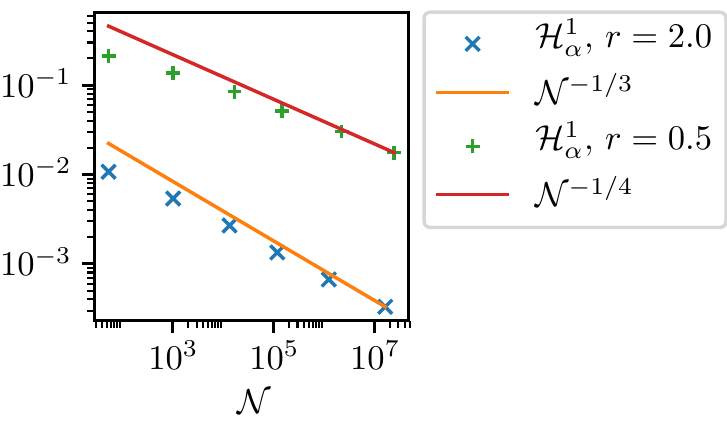}
  \includegraphics{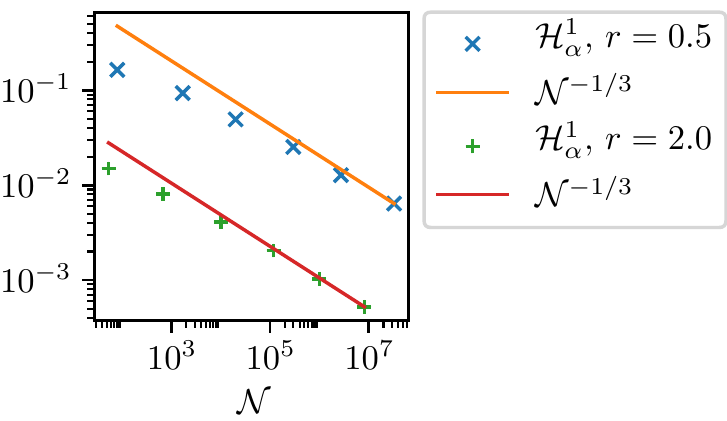}
  \caption{
    \(\mathcal{H}^{1}_{\alpha}\)-error with respect to the total number of degrees of freedom \(\mathcal{N}\) on the unit cube with piecewise linear finite elements (\(k=1\)).
    \(s=0.25\) on the left, \(s=0.75\) on the right.
    Quasi-optimal convergence is obtained.
    (Compare with the optimal order given in \eqref{eq:convN}.)
  }
  \label{fig:error_N_cube}
\end{figure}

\FloatBarrier{}

\section{Conclusion}
\label{sec:conclusion}

A numerical scheme is presented for approximating fractional order Poisson problems in two and three dimensions.
The scheme is based on reformulating the original problem posed over \(\Omega\) on the extruded domain \(\mathcal{C}=\Omega\times[0,\infty)\) following \cite{CaffarelliSilvestre2007_ExtensionProblemRelatedToFractionalLaplacian}.
The resulting degenerate elliptic \emph{integer} order PDE is approximated using a hybrid FEM-spectral scheme.
Finite elements are used on \(\Omega\), whilst an appropriate spectral method is used in the extruded direction.
The spectral part of the scheme requires suitable approximations of the true eigenvalues of the usual Laplacian over \(\Omega\).
We derive an a priori error estimate which takes account of the error arising from the approximation of the true eigenvalues, and present a strategy for choosing suitable approximations of the eigenvalues based on Weyl's law and finite element discretizations of the eigenvalue problem.
The resulting system of linear algebraic equations is decomposed into blocks which are solved using standard iterative solvers such as multigrid and conjugate gradient.
Numerical examples in two and three dimensions show that the approach is quasi-optimal in terms of complexity.

\printbibliography

\end{document}